\documentclass[10pt]{article}
\DeclareMathAlphabet{\mathpzc}{OT1}{pzc}{m}{it}

\usepackage{amsmath}
\usepackage{amssymb}
\usepackage{amsfonts}
\usepackage{amsthm}
\usepackage{mathrsfs}
\usepackage{ifsym}
\usepackage{fontenc}
\usepackage{textcomp}
\usepackage{tipa}
\usepackage{enumerate}
\usepackage[pdftex]{color, graphicx}
\numberwithin{equation}{section}
\usepackage{hyperref}

\begin{document}
 
\title{{\bf  Periodic points of algebraic functions related to a continued fraction of Ramanujan}}       % Enter your title between curly braces
\author{Sushmanth J. Akkarapakam and Patrick Morton}        % Enter your name between curly braces
\date{Sept. 23, 2022}          % Enter your date or \today between curly braces
\maketitle

\begin{abstract}  A continued fraction $v(\tau)$ of Ramanujan is evaluated at certain arguments in the field $K = \mathbb{Q}(\sqrt{-d})$, with $-d \equiv 1$ (mod $8$), in which the ideal $(2) = \wp_2 \wp_2'$ is a product of two prime ideals.  These values of $v(\tau)$ are shown to generate the inertia field of $\wp_2$ or $\wp_2'$ in an extended ring class field over the field $K$.  The conjugates over $\mathbb{Q}$ of these same values, together with $0, -1 \pm \sqrt{2}$, are shown to form the exact set of periodic points of a fixed algebraic function $\hat F(x)$, independent of $d$.  These are analogues of similar results for the Rogers-Ramanujan continued fraction.
\end{abstract}

\section{Introduction}

This paper is concerned with values of Ramanujan's continued fraction
$$v(\tau) = \frac{q^{1/2}}{1+q+} \ \frac{q^2}{1+q^3+} \ \frac{q^4}{1+q^5+} \ \frac{q^6}{1+q^7+}\dots, \ \ q = e^{2\pi i \tau},$$
which is also given by the infinite product
$$v(\tau)=q^{1/2}\prod_{n=1}^\infty\big(1-q^n\big)^{\left(\frac{2}{n}\right)}, \ \ q = e^{2\pi i \tau},$$
for $\tau$ in the upper half-plane.  Here, $\left(\frac{2}{n}\right)$ is the Kronecker symbol.  See \cite{g}, \cite[p. 153]{d}.  The continued fraction $v(\tau)$ is analogous to the Rogers-Ramanujan continued fraction
$$r(\tau) = q^{1/5}\prod_{n=1}^\infty\big(1-q^n\big)^{\left(\frac{5}{n}\right)}\ \ q = e^{2\pi i \tau},$$
whose properties were considered in the papers \cite{m2}, \cite{m3}.  In \cite{m2} it was shown that certain values of $r(\tau)$, for $\tau$ in the imaginary quadratic field $K = \mathbb{Q}(\sqrt{-d})$ with discriminant $-d \equiv \pm 1$ (mod $5$), are periodic points of a fixed algebraic function, independent of $d$, and generate certain class fields $\Sigma_{\mathfrak{f}} \Omega_f$ over $K$.  Here $\Sigma_{\mathfrak{f}}$ is the ray class field of conductor $\mathfrak{f} = \wp_5$ or $\wp_5'$ over $K$, where $(5) = \wp_5 \wp_5'$ in the ring of integers $R_K$ of $K$; and $\Omega_f$ is the ring class field of conductor $f$ corresponding to the order $\textsf{R}_{-d}$ of discriminant $-d= \mathfrak{d}_K f^2$ in $K$ ($\mathfrak{d}_K$ is the discriminant of $K$). \medskip

Here we will show that a similar situation holds for certain values of the continued fraction $v(\tau)$.  We consider discriminants of the form $-d \equiv 1$ (mod $8$) and arguments in the field $K = \mathbb{Q}(\sqrt{-d})$.  Let $R_K$ be the ring of integers in this field and let the prime ideal factorization of (2) in $R_K$ be $(2) = \wp_2 \wp_2'$.  We define the algebraic integer $w$ by
\begin{equation}
w = \frac{a+\sqrt{-d}}{2}, \ \ a^2+d \equiv 0 \ (\textrm{mod} \ 2^5), \ \ (N(w),f) = 1,
\label{eqn:1.1}
\end{equation}
where $\wp_2 = (2,w)$.  Also, the positive (and odd) integer $f$ is defined by $-d = \mathfrak{d}_K f^2$, where $\mathfrak{d}_K$ is the discriminant of $K/\mathbb{Q}$.
\medskip

We will show that
$$v(w/8) = \pm \frac{1 \pm \sqrt{1+\pi^2}}{\pi},$$
where $\pi$ is a generator in $\Omega_f$ of the ideal $\wp_2$ (or rather, its extension $\wp_2 R_{\Omega_f}$ in $\Omega_f$).  The algebraic integer $\pi$ and its conjugate $\xi$ in $\Omega_f$ were studied in \cite{lm} and shown to satisfy
\begin{equation}
\pi^4+\xi^4 = 1, \ \ (\pi) = \wp_2, \ (\xi) = \wp_2', \ \ \xi = \frac{\pi^{\tau^2}+1}{\pi^{\tau^2}-1},
\label{eqn:1.2}
\end{equation}
where $\tau = \left(\frac{\Omega_f/K}{\wp_2}\right)$ is the Artin symbol (Frobenius automorphism) for the prime ideal $\wp_2$ and the ring class field $\Omega_f$ over $K$ whose conductor is $f$.  It follows from results of \cite{lm} that
$$\pi = (-1)^{c} \mathfrak{p}(w),$$
where $c$ is an integer satisfying the congruence
$$c \equiv 1-\frac{a^2+d}{32} \ (\textrm{mod} \ 2)$$
and $\mathfrak{p}(\tau)$ is the modular function $\mathfrak{p}(\tau) = \frac{\mathfrak{f}_2^2(\tau/2)}{\mathfrak{f}^2(\tau/2)}$, defined in terms of the Weber-Schl\"afli functions $\mathfrak{f}_2(\tau), \mathfrak{f}(\tau)$.  (See \cite{we}, \cite{co}, \cite{sch}.)  The above formula for $v(w/8)$ follows from the identity
$$\frac{2}{\mathfrak{p}(8\tau)} = \frac{1-v^2(\tau)}{v(\tau)} = \frac{1}{v(\tau)} - v(\tau),$$
for $\tau$ in the upper half-plane, which we prove in Proposition \ref{prop:5}.  \medskip

As in \cite{m2}, we consider a diophantine equation, namely
$$\mathcal{C}_2: \ X^2+Y^2 = \sigma^2 (1+X^2Y^2), \ \ \sigma = -1+\sqrt{2}.$$
An identity for the continued fraction $v(\tau)$ implies that $(X,Y) = (v(w/8),v(-1/w))$ is a point on $\mathcal{C}_2$.  We prove the following theorem relating the coordinates of this point. \medskip

\noindent {\bf Theorem A.} {\it Let $w$ be given by (\ref{eqn:1.1}) with $\wp_2 = (2,w)$ in $R_K$, and let $-d = \mathfrak{d}_K f^2 \equiv 1$ (mod $8$).}
\begin{enumerate}[(a)]
\item {\it The field $F_1 = \mathbb{Q}(v(w/8)) = \mathbb{Q}(v^2(w/8))$ equals the field $\Sigma_{\mathfrak{\wp_2'^3}} \Omega_f$, where $\Sigma_{\mathfrak{\wp_2'^3}}$ is the ray class field of conductor $\mathfrak{f} = \wp_2'^3$ and $\Omega_f$ is the ring class field of conductor $f$ over the field $K$.  The field $F_1$ is the inertia field for $\wp_2$ in the extended ring class field $L_{\mathcal{O},8} = \Sigma_8 \Omega_f$ over $K$, where $\mathcal{O} = \textsf{R}_{-d}$ is the order of discriminant $-d$ in $K$.}
\item {\it We have $F_2 = \mathbb{Q}(v(-1/w)) = \Sigma_{\mathfrak{\wp_2^3}} \Omega_f$, the inertia field of $\wp_2'$ in $L_{\mathcal{O},8}/K$.}
\item {\it If $\tau_2$ is the Frobenius automorphism $\displaystyle \tau_2 = \left(\frac{F_1/K}{\wp_2}\right)$, then}
\begin{equation}
v(-1/w) =  \frac{v(w/8)^{\tau_2}+(-1)^{c}\sigma}{\sigma v(w/8)^{\tau_2}-(-1)^{c}}.
\label{eqn:1.3}
 \end{equation}
 \end{enumerate}
 \medskip
 
 See Theorems \ref{thm:1}, \ref{thm:3} and \ref{thm:4} and their corollaries. \medskip
 
 From part (c) of this theorem we deduce the following. \medskip
 
 \noindent {\bf Theorem B.} {\it If $w$ and $c$ are as above, then the generator $(-1)^{1+c}v(w/8)$ of the field $\Sigma_{\wp_2'^3} \Omega_f$ over $\mathbb{Q}$ is a periodic point of the multivalued algebraic function $\hat F(x)$ given by 
 \begin{equation}
 \hat F(x) =  -\frac{x^2-1}{2} \pm \frac{1}{2} \sqrt{x^4 - 6x^2 + 1};
 \label{eqn:1.4}
 \end{equation}
 and $v^2(w/8)$ is a periodic point of the algebraic function $\hat T(x)$ defined by}
 \begin{equation}
 \hat T(x) = \frac{1}{2}(x^2-4x+1) \pm \frac{x-1}{2} \sqrt{x^2-6x+1}.
 \label{eqn:1.5}
 \end{equation}
 {\it The minimal period of $(-1)^{1+c}v(w/8)$ (and of $v^2(w/8)$) is equal to the order of the automorphism $\tau_2$ in $\textrm{Gal}(F_1/K)$.  Together with the numbers $0, -1 \pm \sqrt{2}$, the values $(-1)^{1+c}v(w/8)$ and their conjugates over $\mathbb{Q}$ are the only periodic points of the algebraic function $\hat F(x)$ in $\overline{\mathbb{Q}}$ or $\mathbb{C}$.  The only periodic points of $\hat T(x)$ in $\overline{\mathbb{Q}}$ or $\mathbb{C}$ are $0, (-1 \pm \sqrt{2})^2$, and the conjugates of the values $v^2(w/8)$ over $\mathbb{Q}$.}
 \medskip
 
 We understand by a periodic point of the multivalued algebraic function $\hat F(x)$ the following.  Let $f(x,y) = x^2 y+x^2+y^2 - y$ be the minimal polynomial of $\hat F(x)$ over $\mathbb{Q}(x)$.  A periodic point of $\hat F(x)$ is an algebraic  number $a$ for which there exist $a_1, a_2, \dots, a_{n-1} \in \overline{\mathbb{Q}}$ satisfying
 $$f(a,a_1) = f(a_1,a_2) = \cdots = f(a_{n-1}, a) = 0.$$
 See \cite{m1}, \cite{m}.  Thus, if $a \in \overline{\mathbb{Q}}$ is a periodic point of $\hat F(x)$, so are its conjugates over $\mathbb{Q}$, because $f(x,y)$ has coefficients in $\mathbb{Q}$.  We show in Section 8 that $v^2(w/8)$ is actually a periodic point in the usual sense of the single-valued $2$-adic function
 $$T(x) = x^2 - 4x + 2 - (x-1)(x-3) \sum_{k=1}^\infty{C_{k-1} \frac{2^k}{(x-3)^{2k}}},$$
 defined on a subset of the maximal unramified, algebraic extension $\textsf{K}_2$ of the $2$-adic field $\mathbb{Q}_2$.  ($C_{k}$ is the $k$-th Catalan number.)  This follows from the fact that
 $$v(w/8)^{2\tau_2} = T(v(w/8)^2),$$
in the completion $F_{1,\mathfrak{q}} \subset \textsf{K}_2$ of $F_1= \Sigma_{\wp_2'^3} \Omega_f$ with respect to a prime divisor $\mathfrak{q}$ of $\wp_2$ in $F_1$.  This implies that the minimal period of $v^2(w/8)$ with respect to the function $T(x)$ is $n = \textrm{ord}(\tau_2)$. \medskip

These facts are all analogues of the corresponding facts for the Rogers-Ramanujan continued fraction $r(\tau)$ which were proved in \cite{m2} and \cite{m3}. \medskip

An important corollary of the fact that the conjugates of the values $v(w/8)$ in Theorem B are, together with the three fixed points, all the periodic points of the algebraic function $\hat F(x)$, is the following class number formula.  In this formula, $h(-d)$ denotes the class number of the order $\textsf{R}_{-d}$ of discriminant $-d$ in the quadratic field $K = K_d$, and $\mathfrak{D}_{n,2}$ is the finite set of negative discriminants $-d \equiv 1$ (mod $8$) for which the Frobenius automorphism $\tau_2$ in Theorem A has order $n$ in $\textrm{Gal}(F_1/K_d)$ ($F_1 = F_{1,d}$ also depends on $d$):
\begin{equation}
\sum_{-d \in \mathfrak{D}_{n,2}}{h(-d)} = \frac{1}{2} \sum_{k \mid n}{\mu(n/k) 2^k}, \ \ n > 1.
\label{eqn:1.6}
\end{equation}
($\mu(n)$ is the M\"obius function.)  See Theorem \ref{thm:9}.  This fact is the analogue for the prime $p = 2$ of Theorem 1.3 in \cite{m3} for the prime $p=5$, or of Conjecture 1 of that paper for a prime $p > 5$. \medskip

The layout of the paper is as follows.  Section 2 contains a number of $q$-identities (following Ramanujan) and theta function identities which we use to prove identities for various modular functions in Sections 3-5.  In Sections 6 and 7 we prove Theorem A.  The proofs of Theorem B and (\ref{eqn:1.6}) are given in Sections 8 and 9.

\section{Preliminaries.}

As is customary, let us set
$$(a;q)_0:=1,\ \ \ \ \ (a;q)_n:=\prod_{k=0}^{n-1}{(1-aq^k)},\ \ \ \ \ n\ge1$$
and
$$(a;q)_\infty:=\prod_{k=0}^\infty{(1-aq^k)},\ \ \ \ \ |q| \le 1.$$
Ramanujan's general theta function $f(a,b)$ is defined as
\begin{equation}
f(a,b):=\sum_{n=-\infty}^\infty{a^{n(n+1)/2}\,b^{n(n-1)/2}}.
\label{eqn:2.1}
\end{equation}
Three special cases are defined, in Ramanujan's notation, as
\begin{align}
\label{eqn:2.2} \varphi(q):=f(q,q)&=\sum_{n=-\infty}^\infty{q^{n^2}},\\
\label{eqn:2.3} \psi(q):=f(q,q^3)&=\sum_{n=0}^\infty{q^{n(n+1)/2}},\\
\label{eqn:2.4} f(-q):=f(-q,-q^2)&=\sum_{n=-\infty}^\infty{(-1)^n\,q^{n(3n-1)/2}}.
\end{align}
The Jacobi's triple product identity, in Ramanujan's notation, takes the form
\begin{equation}
f(a,b)=(-a;ab)_\infty(-b;ab)_\infty(ab;ab)_\infty.
\label{eqn:2.5}
\end{equation}
Using this, the above three functions can be written as
\begin{align}
\label{eqn:2.6} \varphi(q)&=(-q;q^2)_\infty^2(q^2;q^2)_\infty,\\
\label{eqn:2.7} \psi(q)&=(-q;q)_\infty(q^2;q^2)_\infty=\frac{(q^2;q^2)_\infty}{(q;q^2)_\infty},\\
\label{eqn:2.8} f(-q)&=(q;q)_\infty.
\end{align}
The equality that relates the right hand sides of both the equations for $f(-q)$ in (\ref{eqn:2.4}) and (\ref{eqn:2.8}) is Euler's pentagonal number theorem.\medskip

Another important function that plays a prominent role is given by
\begin{equation}
\chi(q):=(-q;q^2)_\infty.
\label{eqn:2.9}
\end{equation}
All the above four functions satisfy a myriad of relations, most of which are listed and proved in Berndt's books on Ramanujan's notebooks, and we will refer to them as needed.\medskip

Last but not least, the Dedekind-eta function is defined as
\begin{equation}
\eta(\tau)=q^{1/24}\,f(-q),\ \ \ \ \ q=e^{2\pi i\tau},\ \ \ \ \ \textrm{Im} \ \tau>0.
\label{eqn:2.10}
\end{equation}
Most of the identities that we use later on are listed here in order, for the sake of convenience.
\begin{align}
\label{eqn:2.11} \varphi^2(q)+\varphi^2(-q)&=2\varphi^2(q^2),\\
\label{eqn:2.12} \varphi^4(q)-\varphi^4(-q)&=16q\psi^4(q^2),\\
\label{eqn:2.13} \varphi(q)\psi(q^2)&=\psi^2(q),\\
\label{eqn:2.14} \varphi(-q)+\varphi(q^2)&=2\frac{f^2(q^3,q^5)}{\psi(q)},\\
\label{eqn:2.15} \varphi(-q)-\varphi(q^2)&=-2q\frac{f^2(q,q^7)}{\psi(q)},\\
\label{eqn:2.16} \varphi(q)\varphi(-q)&=\varphi^2(-q^2),\\
\label{eqn:2.17} \varphi(q)+\varphi(-q)&=2\varphi(q^4),\\
\label{eqn:2.18} \varphi^2(q)-\varphi^2(-q)&=8q\psi^2(q^4).
\end{align}
\smallskip

\noindent All of these above identities and their proofs can be found in \cite[p. 40, Entry 25]{b1} and in \cite[p. 51, Example (iv)]{b1}.\bigskip

\noindent For $\tau\in\mathcal{H}$, the upper half plane, and $q=e(\tau)=e^{2\pi i\tau}$, the theta constant with characteristic $\big[{\epsilon\atop\epsilon'}\big]\in\mathbb{R}$ is defined as
\begin{equation}
\theta\left[{\epsilon\atop\epsilon'}\right](\tau)= \sum_{n\in\mathbb{Z}}{e\big(\frac{1}{2}(n+\frac{\epsilon}{2})^2\tau+\frac{\epsilon'}{2}(n+\frac{\epsilon}{2})\big)}.
\label{eqn:2.19}
\end{equation}
It satisfies the following basic properties for $l,m,n\in\mathbb{Z}$ with $N$ positive:
\begin{align}
\label{eqn:2.20} \theta\left[{\epsilon\atop\epsilon'}\right](\tau)&= e\big(\mp\frac{\epsilon m}{2}\big)\ \theta\left[{\pm\epsilon+2l\atop\pm\epsilon'+2m}\right](\tau),\\
\label{eqn:2.21} \theta\left[{\epsilon\atop\epsilon'}\right](\tau)&=\sum_{k=0}^{N-1}{\theta\left[{\frac{\epsilon+2k}{N}\atop N\epsilon'}\right]\big(N^2\tau\big)}.
\end{align}
We also have the transformation law, for $\bigl( \begin{smallmatrix}a&b\\c&d\end{smallmatrix}\bigr)\in\rm{SL}_2(\mathbb{Z})$:
\begin{equation}
\theta\left[{\epsilon\atop\epsilon'}\right] \big(\frac{a\tau+b}{c\tau+d}\big)=\kappa\sqrt{c\tau+d}\ \theta\left[{a\epsilon+c\epsilon'-ac\atop b\epsilon+d\epsilon'+bd}\right](\tau),
\label{eqn:2.22}
\end{equation}
where
$$\kappa=\textstyle e\big(-\frac{1}{4}(a\epsilon+c\epsilon')bd-\frac{1}{8}(ab\epsilon^2+cd\epsilon'^2+2bc\epsilon\epsilon')\big)\kappa_0,$$
with $\kappa_0$ an eighth root of unity depending only on the matrix $\bigl( \begin{smallmatrix}a&b\\c&d\end{smallmatrix}\bigr)$.\medskip

\noindent In particular, we have:
\begin{align}
\label{eqn:2.23} \theta\left[{\epsilon\atop\epsilon'}\right](\tau+1)&=\textstyle e\big(-\frac{\epsilon}{4}(1+\frac{\epsilon}{2})\big)\ \displaystyle\theta\left[{\epsilon\atop\epsilon+\epsilon'+1}\right](\tau),\\
\label{eqn:2.24} \theta\left[{\epsilon\atop\epsilon'}\right]\textstyle\big(\frac{-1}{\tau}\big)&=\textstyle e\big(-\frac{1}{8}\big)\ \sqrt{\tau}\ e\big(\frac{\epsilon\epsilon'}{4}\big)\ \displaystyle\theta\left[{\epsilon'\atop-\epsilon}\right](\tau).
\end{align}
We also have the product formula:
\begin{equation}
\theta\left[{\epsilon\atop\epsilon'}\right](\tau)= e\big(\frac{\epsilon\epsilon'}{4}\big) q^{\frac{\epsilon^2}{8}}\prod_{n\ge1}{\big(1-q^n\big)\big(1+e\left(\frac{\epsilon'}{2}\right) q^{n-\frac{1+\epsilon}{2}}\big) \big(1+e\left(\frac{-\epsilon'}{2}\right) q^{n-\frac{1-\epsilon}{2}}\big)},
\label{eqn:2.25}
\end{equation}
which follows from the Jacobi's triple product identity.\medskip

\noindent More information about these theta constants and the above formulas, as well as their proofs, can all be found in \cite[pp. 71-81]{fk}. Or also see \cite[pp. 143, 158-159]{d}.
\medskip

\section{Identities for $u(\tau)$ and $v(\tau)$}

Let us define the functions $u(\tau)$ and $v(\tau)$ as
$$u(\tau)=\sqrt{2}\,q^{1/8}\prod_{n=1}^\infty{\big(1+q^n\big)^{(-1)^n}},$$
$$v(\tau)=q^{1/2}\prod_{n=1}^\infty{\big(1-q^n\big)^{\left(\frac{8}{n}\right)}}.$$\medskip

The functions $u(\tau)$ and $v(\tau)$ satisfy the following identities. \bigskip

\newtheorem{prop}{Proposition}

\begin{prop}
(a) If $x = u(\tau)$ and $y = u(2\tau)$, we have
$$x^4(y^4+1) = 2y^2.$$
(b) If $x = v(\tau)$ and $y = v(2\tau)$, we have
$$x^2 y+x^2+y^2 = y.$$
\label{prop:1}
\end{prop}

\noindent {\bf Remark.} The curve $E: f(x,y) = 0$ defined by
$$f(x,y) = x^2 y+x^2+y^2 - y$$
is an elliptic curve with $j(E) = 1728$, so $E$ has complex multiplication by $\textsf{R} = \mathbb{Z}[i]$.

\begin{proof}
(a) From (\ref{eqn:2.11}), we have
$$\varphi^2(-q)=2\varphi^2(q^2)-\varphi^2(q),$$
where
\begin{equation*}
\varphi(q) =(-q;q^2)_\infty^2(q^2;q^2)_\infty \ \ \textrm{and} \ \ \psi(q) =\frac{(q^2;q^2)_\infty}{(q;q^2)_\infty}
\end{equation*}
are as defined in (\ref{eqn:2.6}) and (\ref{eqn:2.7}).  Squaring both sides gives us
$$\varphi^4(-q)=4\varphi^4(q^2)-4\varphi^2(q)\varphi^2(q^2)+\varphi^4(q).$$
Using
$$\varphi^4(q)-\varphi^4(-q)=16q\psi^4(q^2),$$
which is (\ref{eqn:2.12}), we obtain
$$\varphi^4(q^2)+4q\,\psi^4(q^2)=\varphi^2(q)\,\varphi^2(q^2).$$
Dividing both sides by $\varphi^4(q^2)$ and using the relation $\psi^2(q)=\varphi(q)\,\psi(q^2)$ from (\ref{eqn:2.13})
we get
\begin{equation}
1+4q\frac{\psi^4(q^2)}{\varphi^4(q^2)} = \frac{\varphi^2(q)}{\varphi^2(q^2)} =\frac{\psi^2(q^2)}{\varphi^2(q^2)} \cdot \frac{\varphi^4(q)}{\psi^4(q)}.
\label{eqn:3.1}
\end{equation}
Since
$$u(\tau)=\sqrt{2}q^{1/8}\prod_{n=1}^\infty{(1+q^n)^{(-1)^n}}=\sqrt{2}q^{1/8}\frac{(-q^2;q^2)_\infty}{(-q;q^2)_\infty}=\sqrt{2}q^{1/8}\frac{\psi(q)}{\varphi(q)},$$
the result follows by substituting the last equality for $u(\tau)$ into (\ref{eqn:3.1}). \medskip

\noindent(b) From \cite[p. 153, (9.7)]{d} we have the following relation between $u = u(\tau)$ and $v = v(\tau)$:
\begin{equation}
u^4(v^2+1)^2+4v(v^2-1)=0;
\label{eqn:3.2}
\end{equation}
which can be rewritten as $\displaystyle u^4=\frac{4v(1-v^2)}{(v^2+1)^2}$.  Substituting this expression for $u^4$ into the relation $u^4(\tau)\big[u^4(2\tau)+1\big]=2u^2(2\tau)$, after squaring, we obtain
$$\frac{16x^2(1-x^2)^2}{(x^2+1)^4}\cdot\bigg[\frac{4y(1-y^2)}{(y^2+1)^2}+1\bigg]^2=4\cdot\frac{4y(1-y^2)}{(y^2+1)^2},$$
where $x=v(\tau), y=v(2\tau)$. Clearing the denominators gives us
$$x^2(1-x^2)^2(y^2-2y-1)^4=y(1-y^2)(y^2+1)^2(x^2+1)^4.$$
Now moving everything to one side and factoring the polynomial using Maple, we finally arrive at
\begin{align*}
(x^2y+x^2+y^2-y)&(x^2y^2-x^2y+y+1)(x^2y^2+2xy^2+x^2-4xy+y^2-2x+1)\\
& \times (x^2y^2-2xy^2+x^2+4xy+y^2+2x+1)=0.
\end{align*}
From the definitions of $x$ and $y$, it is clear that $x=O(q^{1/2})$ and $y=O(q)$ as $q$ tends to $0$. Hence, the first factor above (and none of the others) vanishes for $q$ sufficiently small.  By the identity theorem, the first factor vanishes for $|q|<1$. This proves the result.
\end{proof}

\noindent {\bf Remark.}  The identity in part (b) of Proposition \ref{prop:1} can be written as
$$v^2(\tau) = v(2\tau) \frac{1-v(2\tau)}{1+v(2\tau)}.$$
This is analogous to the identity for the Rogers-Ramanujan continued fraction $r(\tau)$:
$$r^5(\tau) = r(5\tau) \frac{r^4(5\tau)-3r^3(5\tau) +4r^2(5\tau) -2r(5\tau) +1}{r^4(5\tau) +2r^3(5\tau) +4r^2(5\tau) +3r(5\tau) +1}.$$
See \cite[p. 167]{b3}, \cite[pp. 19-20]{b2}.
\medskip

\begin{prop} The functions $x=v^2(\tau)$ and $y=v^2(2\tau)$ satisfy the relation
$$g(x,y) = y^2-(x^2-4x+1)y+x^2 = 0.$$
\label{prop:2}
\end{prop}
\begin{proof}
For $x=v(\tau)$ and $y=v(2\tau)$, we have the relation
$$x^2+y^2=y (1-x^2).$$
Squaring both sides and moving all the terms to the left side, we obtain
$$x^4+y^4+4x^2y^2-x^4y^2-y^2=0.$$
Hence, $x=v^2(\tau)$ and $y=v^2(2\tau)$ satisfy the relation
$$g(x,y)=x^2+y^2+4xy-x^2y-y=0.$$
\end{proof}

Let $A, \bar A$ denote the linear fractional mappings
\begin{equation}
A(x) = \frac{\sigma x +1}{x-\sigma}, \ \ \bar A(x) = \frac{-x+\sigma}{\sigma x+1}, \ \ \sigma = -1+\sqrt{2}.
\label{eqn:3.3}
\end{equation}

\begin{prop} The following identity holds:
\begin{align}
\notag v\left(\frac{-1}{\tau}\right) &= \bar A(v(\tau/4)) = \frac{\bar \sigma v(\tau/4) + 1}{v(\tau/4) - \bar \sigma} = \frac{-v(\tau/4)+\sigma}{\sigma v(\tau/4)+1},
\end{align}
where $\bar \sigma = -1-\sqrt{2}$.
\label{prop:3}
\end{prop}
\begin{proof} This follows from the formula
\begin{equation*}
v(\tau) = e^{-2 \pi i/8} \frac{\theta [{3/4 \atop 1}](8\tau)}{\theta [{1/4 \atop 1}](8\tau)},
\end{equation*}
using the formulas (\ref{eqn:2.20}), (\ref{eqn:2.21}), (\ref{eqn:2.24}).  (Also see \cite{fk}.)  Namely, we have:
$$v\bigg(\frac{-1}{\tau}\bigg)=e^{-2\pi i/8}\frac{\theta\big[{3/4\atop1}\big]\big(\frac{-8}{\tau}\big)}{\theta\big[{1/4\atop1}\big]\big(\frac{-8}{\tau}\big)}=\frac{\theta\big[{1\atop3/4}\big]\big(\frac{\tau}{8}\big)}{\theta\big[{1\atop1/4}\big]\big(\frac{\tau}{8}\big)}=\frac{\sum\limits_{k=0}^3\theta\left[{\frac{1+2k}{4}\atop3}\right]\big(2\tau)}{\sum\limits_{k=0}^3\theta\left[{\frac{1+2k}{4}\atop1}\right](2\tau)},$$
which after some simplification yields
$$v\bigg(\frac{-1}{\tau}\bigg)=\frac{\big[-1+e^{3\pi i/8}\big]v(\tau/4)+\big[e^{2\pi i/8}+e^{3\pi i/2}\big]}{\big[e^{2\pi i/8}+e^{3 \pi i/2}\big]v(\tau/4)+\big[1+e^{7\pi i/8}\big]}.$$
This yields that
$$v\bigg(\frac{-1}{\tau}\bigg)=\frac{\bar \sigma v(\tau/4)+1}{v(\tau/4)-\bar \sigma} = \frac{-v(\tau/4)+\sigma}{\sigma v(\tau/4)+1}.$$
\end{proof}

The set of mappings
$$\tilde H = \{x, A(x), \bar A(x), -1/x\}$$
forms a group under composition.  We also have the formula
$$(\sigma x+1)^2(\sigma y+1)^2 f(\bar A(x),\bar A(y)) = 2^3 \sigma^2 f(y,x).$$

\begin{prop} The function $v(\tau)$ satisfies the following:
\begin{equation}
v^2\left(\frac{-1}{8\tau}\right) = \frac{v^2(\tau)-\sigma^2}{\sigma^2 v^2(\tau)-1}, \ \ \sigma = -1+ \sqrt{2};\\
\label{eqn:3.4}
\end{equation}
\label{prop:4}
\end{prop}
\begin{proof}
Replacing $\tau$ by $8\tau$ in Proposition \ref{prop:3} and squaring gives us
\begin{align*}
v^2\left(\frac{-1}{8\tau}\right)&=\frac{(-v(2\tau)+\sigma)^2}{(\sigma v(2\tau)+1)^2}\\
&=\frac{(-y+\sigma)^2}{(\sigma y+1)^2}\\
&=\frac{y^2-2\sigma y+\sigma^2}{\sigma^2y^2+2\sigma y+1},
\end{align*}
where $y=v(2\tau)$.  Then, replace $2\sigma$ by $1-\sigma^2$ to obtain
\begin{align*}
v^2\left(\frac{-1}{8\tau}\right)&=\frac{y^2-y+\sigma^2y+\sigma^2}{\sigma^2y^2+y-\sigma^2y+1}\\
&=\frac{\sigma^2(y+1)-(y-y^2)}{(y+1)-\sigma^2(y-y^2)}.
\end{align*}
Now finally, replace $(y-y^2)$ by $x^2(y+1)$, using Proposition \ref{prop:1}(b), to get the result:
\begin{align*}
v^2\left(\frac{-1}{8\tau}\right)&=\frac{\sigma^2(y+1)-x^2(y+1)}{(y+1)-\sigma^2x^2(y+1)}\\
&=\frac{(\sigma^2-x^2)(y+1)}{(1-\sigma^2x^2)(y+1)}\\
&=\frac{x^2-\sigma^2}{\sigma^2x^2-1},
\end{align*}
where $x=v(\tau)$.  This completes the proof.
\end{proof}

For later use we denote the linear fractional map which occurs in (\ref{eqn:3.4}) by $t(x)$:
\begin{equation}
t(x) = \frac{x-\sigma^2}{\sigma^2 x-1}.
\label{eqn:3.5}
\end{equation}
A straightforward calculation shows that
\begin{equation}
(\sigma^2 x-1)^2 (\sigma^2 y - 1)^2 g(t(x),t(y)) = 2^5 \sigma^4 g(y,x).
\label{eqn:3.6}
\end{equation}

\section{The relation between $v(\tau)$ and $\mathfrak{p}(\tau)$.}

In this section and the next we shall prove several identities between $v(\tau)$ and the functions $\mathfrak{p}(\tau)$ and $\mathfrak{b}(\tau)$ defined as follows.  Let $\mathfrak{f}, \mathfrak{f}_1, \mathfrak{f}_2$ denote the Weber-Schl\"afli functions (see \cite[p. 233]{co}, \cite[p. 148]{sch}).  Then the functions $\mathfrak{p}(\tau)$ and $\mathfrak{b}(\tau)$ are given by

\begin{align}
\label{eqn:4.1} \mathfrak{p}(\tau) & = \frac{\mathfrak{f}_2(\tau/2)^2}{\mathfrak{f}(\tau/2)^2} = 2q^{1/16} \prod_{n=1}^\infty{\left(\frac{1+q^{n/2}}{1+q^{n/2-1/4}}\right)^2},\\
\label{eqn:4.2} \mathfrak{b}(\tau) & = 2\frac{\mathfrak{f}_1(\tau/2)^2}{\mathfrak{f}(\tau/2)^2} = 2\prod_{n=1}^\infty{\left(\frac{1-q^{n/2-1/4}}{1+q^{n/2-1/4}}\right)^2}.
\end{align}
\medskip
Note that $\mathfrak{b}(\tau)$ occurs in \cite[\S 10, (10.3)]{lm}.

\begin{prop}
\begin{equation}
\frac{2}{\mathfrak{p}(8\tau)} = \frac{1-v^2(\tau)}{v(\tau)} = \frac{1}{v(\tau)} - v(\tau).
\label{eqn:4.3}
\end{equation}
\label{prop:5}
\end{prop}

\begin{proof}
The function $v(\tau)$ satisfies
\begin{align*}
v(\tau)&=q^{1/2} \prod_{n\ge1}{(1-q^n)^{\left(\frac{8}{n}\right)}}=q^{1/2}\prod_{n\ge1}{\frac{(1-q^{8n-1})(1-q^{8n-7})}{(1-q^{8n-3})(1-q^{8n-5})}}\\
&=q^{1/2} \frac{(q;q^8)_\infty(q^7;q^8)_\infty}{(q^3;q^8)_\infty(q^5;q^8)_\infty}.
\end{align*}
This gives that
\begin{align*}
\frac{1}{v(\tau)}-v(\tau)&=q^{-1/2}\,\frac{(q^3;q^8)_\infty(q^5;q^8)_\infty}{(q;q^8)_\infty(q^7;q^8)_\infty}-q^{1/2}\,\frac{(q;q^8)_\infty(q^7;q^8)_\infty}{(q^3;q^8)_\infty(q^5;q^8)_\infty}\\
&=\frac{(q^3;q^8)_\infty^2(q^5;q^8)_\infty^2-q\,(q;q^8)_\infty^2(q^7;q^8)_\infty^2}{q^{1/2}\,(q;q^8)_\infty(q^3;q^8)_\infty(q^5;q^8)_\infty(q^7;q^8)_\infty}\\
&=\frac{(q^3;q^8)_\infty^2(q^5;q^8)_\infty^2-q\,(q;q^8)_\infty^2(q^7;q^8)_\infty^2}{q^{1/2}\,(q;q^2)_\infty}.
\end{align*}
Multiplying the numerator and the denominator by $(q^8;q^8)_\infty^2$ and applying Jacobi's triple product identity in the form
$$f(a,b)=(-a;ab)_\infty(-b;ab)_\infty(ab;ab)_\infty,$$
with $(a,b)=(-q^3,-q^5)$ for the first term in the numerator and $(a,b)=(-q,-q^7)$ for the second, we obtain
\begin{align*}
\frac{1}{v(\tau)}-v(\tau)&=\frac{(q^3;q^8)_\infty^2(q^5;q^8)_\infty^2(q^8;q^8)_\infty^2-q\,(q;q^8)_\infty^2(q^7;q^8)_\infty^2(q^8;q^8)_\infty^2}{q^{1/2}\,(q;q^2)_\infty(q^8;q^8)_\infty^2}\\
&=\frac{f^2(-q^3,-q^5)-q\,f^2(-q,-q^7)}{q^{1/2}\,(q;q^2)_\infty(q^8;q^8)_\infty^2}.
\end{align*}
Now replace $q$ by $-q$ in (\ref{eqn:2.14}), (\ref{eqn:2.15}) and apply it to the numerator to get
\begin{align*}
\frac{1}{v(\tau)}-v(\tau)&=\frac{\psi(-q)\big[\varphi(q)+\varphi(q^2)\big]-\psi(-q)\big[\varphi(q)-\varphi(q^2)\big]}{2\,q^{1/2}\,(q;q^2)_\infty(q^8;q^8)_\infty^2}\\
&=\frac{\psi(-q)\times\varphi(q^2)}{q^{1/2}\,(q;q^2)_\infty(q^8;q^8)_\infty^2}\\
&=q^{-1/2}\,\frac{(q^2;q^2)_\infty}{(-q;q^2)_\infty}\times\frac{(-q^2;q^4)_\infty^2(q^4;q^4)_\infty}{(q;q^2)_\infty(q^8;q^8)_\infty^2}\\&=q^{-1/2}\,\frac{(-q^2;q^4)_\infty^2(q^2;q^2)_\infty(q^4;q^4)_\infty}{(q^2;q^4)_\infty(q^8;q^8)_\infty^2}\\
&=q^{-1/2}\,\frac{(-q^2;q^4)_\infty^2(q^4;q^4)_\infty^2}{(q^8;q^8)_\infty^2}\\&=q^{-1/2}\,(-q^2;q^4)_\infty^2(q^4;q^8)_\infty^2\\
&=q^{-1/2}\,\frac{(-q^2;q^4)_\infty^2}{(-q^4;q^4)_\infty^2}.
\end{align*}
Since
$$\mathfrak{p}(8\tau)=2\,q^{1/2}\prod_{n\ge1}{\bigg(\frac{1+q^{4n}}{1+q^{4n-2}}\bigg)^2}=2\,q^{1/2}\,\frac{(-q^4;q^4)_\infty^2}{(-q^2;q^4)_\infty^2},$$
we get the result by substituting into the last equality.
\end{proof}

\begin{prop}
The function $\mathfrak{p}(\tau)$ satisfies the identity
\begin{equation*}
\mathfrak{p}^2(\tau) \mathfrak{p}^2(2\tau) + \mathfrak{p}^2(\tau) -2\mathfrak{p}(2\tau) = 0.
\end{equation*}
\label{prop:6}
\end{prop}
\begin{proof}
We use the relation between $x=v(\tau)$ and $y=v(2\tau)$ from Proposition \ref{prop:1}(b): $x^2=\frac{y(1-y)}{(1+y)}$.
This gives
\begin{align*}
\left(\frac{2x}{1-x^2}\right)^2&=\frac{4x^2}{(1-x^2)^2}=\frac{4\cdot\frac{y(1-y)}{(1+y)}}{\big(1-\frac{y(1-y)}{(1+y)}\big)^2}\\
&=\frac{4y(1-y)(1+y)}{\big((1+y)-y(1-y)\big)^2}\\
&=\frac{4y(1-y^2)}{(1+y^2)^2}\\
&=\frac{4y(1-y^2)}{4y^2+(1-y^2)^2}.
\end{align*}
Now divide both the numerator and the denominator by $(1-y^2)^2$ to obtain
\begin{equation}
\left(\frac{2x}{1-x^2}\right)^2=\frac{\frac{4y}{1-y^2}}{\frac{4y^2}{(1-y^2)^2}+1}=\frac{2\cdot\big(\frac{2y}{1-y^2}\big)}{\big(\frac{2y}{1-y^2}\big)^2+1}.
\label{eqn:4.4}
\end{equation}
From Proposition \ref{prop:5}, we know that
$$\mathfrak{p}(8\tau)=\frac{2v(\tau)}{1-v^2(\tau)}=\frac{2x}{1-x^2},$$
and
$$\mathfrak{p}(16\tau)=\frac{2v(2\tau)}{1-v^2(2\tau)}=\frac{2y}{1-y^2}.$$
Thus, (\ref{eqn:4.4}) becomes
$$\mathfrak{p}^2(8\tau)=\frac{2\mathfrak{p}(16\tau)}{\mathfrak{p}^2(16\tau)+1}.$$
Replacing $\tau$ by $\tau/8$ and rearranging gives us the result.
\end{proof}

\begin{prop} a) The functions $x=\mathfrak{b}(\tau)$ and $y=\mathfrak{b}(2\tau)$ satisfy the relation
$$x^2y^2+4y^2-16x = 0.$$
b) The following identity holds between $x = \mathfrak{b}(\tau)$ and $z = \mathfrak{b}(4\tau)$:
$$(\mathfrak{b}(\tau)+2)^4 \mathfrak{b}^4(4\tau) = 2^8(\mathfrak{b}^3(\tau)+4\mathfrak{b}(\tau)).$$
\label{prop:7}
\end{prop}
\begin{proof}
a) On putting $4\tau$ for $\tau$ in $x$, we have
$$\mathfrak{b}(4\tau)=2\prod_{n=1}^\infty{\left(\frac{1-q^{2n-1}}{1+q^{2n-1}}\right)^2}=2\ \frac{(q;q^2)_\infty^2}{(-q;q^2)_\infty^2}=2\ \frac{\varphi(-q)}{\varphi(q)}\cdot$$
From (\ref{eqn:2.11}), we have
$$\varphi^2(-q)+\varphi^2(q)=2\varphi^2(q^2).$$
Multiplying both sides by $\varphi^2(-q^2)=\varphi(q)\varphi(-q)$ from (\ref{eqn:2.16}), we obtain
$$\varphi^2(-q)\varphi^2(-q^2)+\varphi^2(q)\varphi^2(-q^2)=2\varphi(q)\varphi(-q)\varphi^2(q^2).$$
Now dividing both sides by $\varphi^2(q)\varphi^2(q^2)$ gives us
$$\frac{\varphi^2(-q)}{\varphi^2(q)}\cdot\frac{\varphi^2(-q^2)}{\varphi^2(q^2)}+\frac{\varphi^2(-q^2)}{\varphi^2(q^2)}=2\ \frac{\varphi(-q)}{\varphi(q)}.$$
Hence, we see that $x=\mathfrak{b}(4\tau)$ and $y=\mathfrak{b}(8\tau)$ satisfy the relation
$$x^2y^2+4y^2-16x=0.$$
Now replace $\tau$ by $\tau/4$. \medskip

b) From (\ref{eqn:2.17}), upon taking fourth powers, we get
$$\big[\varphi(-q)+\varphi(q)\big]^4=16\,\varphi^4(q^4).$$
Multiplying both sides by $\varphi^4(-q^4)/\big[\varphi^4(q)\varphi^4(q^4)\big]$ gives us
$$\frac{\big[\varphi(-q)+\varphi(q)\big]^4}{\varphi^4(q)} \cdot \frac{\varphi^4(-q^4)}{\varphi^4(q^4)}=16\,\frac{\varphi^4(-q^4)}{\varphi^4(q)}.$$
Then using (\ref{eqn:2.16}) twice for the right side, we obtain
$$\frac{\big[\varphi(-q)+\varphi(q)\big]^4}{\varphi^4(q)}\cdot\frac{\varphi^4(-q^4)}{\varphi^4(q^4)}=16\,\frac{\varphi(-q)\varphi(q)}{\varphi^4(q)}\cdot\varphi^2(q^2).$$
Now use (\ref{eqn:2.11}) for the last factor on the right side to get
$$\frac{\big[\varphi(-q)+\varphi(q)\big]^4}{\varphi^4(q)}\cdot\frac{\varphi^4(-q^4)}{\varphi^4(q^4)}=8\,\frac{\varphi(-q)}{\varphi^3(q)}\cdot\big[\varphi^2(-q)+\varphi^2(q)\big].$$
This implies that
$$\bigg[\frac{\varphi(-q)}{\varphi(q)}+1\bigg]^4\cdot\bigg[\frac{\varphi(-q^4)}{\varphi(q^4)}\bigg]^4=8\cdot\frac{\varphi(-q)}{\varphi(q)}\cdot\bigg[\bigg(\frac{\varphi(-q)}{\varphi(q)}\bigg)^2+1\bigg].$$
The result follows on multiplying through by $2^8$ and substituting $\displaystyle\mathfrak{b}(4\tau)=2 \frac{\varphi(-q)}{\varphi(q)}$ and $\displaystyle\mathfrak{b}(16\tau)=2 \frac{\varphi(-q^4)}{\varphi(q^4)}$ into the above equation, and then replacing $\tau$ by $\tau/4$.
\end{proof}

\section{The relation between $v(\tau)$ and $\mathfrak{b}(\tau)$.}

We begin this section by proving the following identity.

\begin{prop}
\begin{equation}
\frac{(v^2(\tau)+1)^2}{v^4(\tau)-6v^2(\tau)+1}=\frac{4}{\mathfrak{b}^2(4\tau)}.
\label{eqn:5.1}
\end{equation}
\label{prop:8}
\end{prop}

\begin{proof}
We prove (\ref{eqn:5.1}) using the identity relating the Weber-Schl\"afli functions from \cite[p. 86, (12)]{we} (see also \cite[p. 234, (12.18)]{co}):
$$\mathfrak{f}_1^8(\tau)+\mathfrak{f}_2^8(\tau)=\mathfrak{f}^8(\tau).$$
From the definitions (\ref{eqn:4.1}) and (\ref{eqn:4.2}) of $\mathfrak{p}(\tau)$ and $\mathfrak{b}(\tau)$, this identity translates to
$$\frac{\mathfrak{b}^4(4\tau)}{16}=1-\mathfrak{p}^4(4\tau).$$
Using the result of Proposition \ref{prop:5}, we write this equation as
\begin{equation*}
\frac{\mathfrak{b}^4(4\tau)}{16}=1-\left(\frac{2\,v(\tau/2)}{1-v^2(\tau/2)}\right)^4=1-\frac{16\,v^4(\tau/2)}{\big(1-v^2(\tau/2)\big)^4}.
\end{equation*}
Setting $x = v(\tau/2)$ and $y = v(\tau)$ and using the relation between $x$ and $y$ from Proposition \ref{prop:1}(b) in the form
$x^2=\frac{y(1-y)}{(1+y)}$ gives that
\begin{align*}
\frac{\mathfrak{b}^4(4\tau)}{16}&=1-\frac{16\,x^4}{(1-x^2)^4} =1-\frac{16\,\left(\frac{y(1-y)}{(1+y)}\right)^2}{\left(1-\frac{y(1-y)}{(1+y)}\right)^4}\\
&=1-\frac{16\,y^2(1-y^2)^2}{(1+y^2)^4}=\frac{(y^2+1)^4-16\,y^2(y^2-1)^2}{(y^2+1)^4}\\
&=\frac{\big((y^2-1)^2+4y^2\big)^2-16\,y^2(y^2-1)^2}{(y^2+1)^4}\\
&=\frac{\big((y^2-1)^2-4y^2\big)^2}{(y^2+1)^4}\\
&=\frac{(y^4-6y^2+1)^2}{(y^2+1)^4},
\end{align*}
which is equivalent to (\ref{eqn:5.1}).  (The plus sign holds on taking the square-root because $\mathfrak{b}(i\infty) = 2, v^2(i\infty) = 0$.)
\end{proof}
\medskip

Proposition \ref{prop:8} will now be used to prove the following formula for the function $j(\tau)$ in terms of $v(\tau)$.

\begin{prop} If $v = v(\tau)$ and $\tau$ lies in the upper half-plane, we have
$$j(\tau) = \frac{(v^{16} + 232v^{14} + 732v^{12}- 1192v^{10} + 710v^8 - 1192v^6 + 732v^4 + 232v^2 + 1)^3}{v^2(v^2 - 1)^2(v^2 + 1)^4(v^4 - 6v^2 + 1)^8}.$$
\label{prop:9}
\end{prop}

\begin{proof}
Let
$$G(x) = \frac{(x^2-16x+16)^3}{x(x-16)}.$$
Then from \cite[p. 1967, (2.8)]{lm} the function
\begin{equation}
\alpha(\tau) = \zeta_8^{-1} \frac{\eta(\tau/4)^2}{\eta(\tau)^2}, \ \ \zeta_8 = e^{2\pi i/8},
\label{eqn:5.2}
\end{equation}
satisfies the following relation:
\begin{equation}
j(\tau) = \frac{(\alpha^8-16\alpha^4+16)^3}{\alpha^4(\alpha^4-16)} = G(\alpha^4(\tau)).
\label{eqn:5.3}
\end{equation}
Moreover, $\alpha(\tau)$ and $\mathfrak{b}(\tau)$ satisfy
$$16 \alpha^4(\tau) + 16 \mathfrak{b}^4(\tau) = \alpha^4(\tau) \mathfrak{b}^4(\tau),$$
so that
\begin{equation}
\alpha^4(\tau) = \frac{16\mathfrak{b}^4(\tau)}{\mathfrak{b}^4(\tau)-16}.
\label{eqn:5.4}
\end{equation}
Setting $b=\mathfrak{b}(\tau)$, we substitute for $\alpha = \alpha(\tau)$ in (\ref{eqn:5.3}) and find that
$$j(\tau) = G\left(\frac{16b^4}{b^4-16}\right) = \frac{(b^8+224b^4+256)^3}{b^4(b^4-16)^4}, \ \ b = \mathfrak{b}(\tau).$$
Now replace $\tau$ by $4\tau$ and use (\ref{eqn:5.1}) to replace $\mathfrak{b}^4(4\tau)$ by
$$\mathfrak{b}^4(4\tau) = \frac{16(v^4-6v^2+1)^2}{(v^2+1)^4},$$
giving
\begin{equation}
j(4\tau) = \frac{(v^{16}-8v^{14}+12v^{12}+8v^{10}+230v^8+8v^6+12v^4-8v^2+1)^3}{v^8(v^2+1)^4(v^2-1)^8 (v^4-6v^2+1)^2},
\label{eqn:5.5}
\end{equation}
with $v = v(\tau)$.  Now replace $v(\tau)$ by $\bar A(v(-1/4\tau))$ from Proposition \ref{prop:3}.  This gives that
$$j(4\tau) = j_2(x^2),$$
where $x = v(-1/4\tau)$ and $j_2(x)$ is the rational function
\begin{equation}
j_2(x) = \frac{(x^{8} + 232x^{7} + 732x^{6}- 1192x^{5} + 710x^4 - 1192x^3 + 732x^2 + 232x + 1)^3}{x(x - 1)^2(x + 1)^4(x^2 - 6x+ 1)^8}.
\label{eqn:5.6}
\end{equation}
Finally, replace $\tau$ by $\tau/4$ to give that
$$j(\tau) = j_2(v^2(-1/\tau)),$$
which implies that $j_2(v^2(\tau)) = j(-1/\tau) = j(\tau)$, completing the proof.
\end{proof}

We highlight the relation
\begin{equation}
j(\tau) = j_2(v^2(\tau)),
\label{eqn:5.7}
\end{equation}
which we will make use of in Section 7.  Using the linear fractional map $t(x)$ from (\ref{eqn:3.5}) and the identity $v^2(-1/8\tau) = t(v^2(\tau))$ in (\ref{eqn:3.4}) yields
$$j\left(\frac{-1}{8\tau}\right) = j_{2}\left(v^2\left(\frac{-1}{8\tau}\right)\right) = j_2(t(v^2(\tau))).$$
A calculation on Maple shows that
$$j_{22}(x) = j_2(t(x)) = \frac{(x^8 - 8x^7 + 12x^6 + 8x^5 - 10x^4 + 8x^3 + 12x^2 - 8x + 1)^3}{x^8(x - 1)^4(x + 1)^2(x^2 - 6x + 1)}.$$
Therefore,
\begin{equation}
j\left(\frac{-1}{8\tau}\right) =  j_{22}(v^2(\tau)).
\label{eqn:5.8}
\end{equation}
Since $\Phi_8(j(\tau), j(-1/8\tau)) = 0$, elliptic curves with these two $j$-invariants are connected by an isogeny of degree $8$.  Hence, the values $j_2(x)$ and $j_{22}(x)$ are the $j$-invariants of two isogenous elliptic curves.  \medskip

We take this opportunity to prove the following known identity (see \cite[p. 154]{d}) from the results we have established so far.

\begin{prop}
\begin{equation}
v^{-2}(\tau) + v^2(\tau) - 6 = \frac{\eta^4(\tau) \eta^2(4\tau)}{\eta^2(2\tau) \eta^4(8\tau)}.
\label{eqn:5.9}
\end{equation}
\label{prop:10}
\end{prop}

\begin{proof}
We will show that (\ref{eqn:5.9}) follows from (\ref{eqn:5.1}).  We first have that
\begin{align*}
v^{-2}(\tau)+v^2(\tau)-6&=\frac{v^4(\tau)-6v^2(\tau)+1}{v^2(\tau)}\\&=\frac{8}{\left(\frac{(v^2(\tau)+1)^2}{v^4(\tau)-6v^2(\tau)+1}\right)-1}\\
&=\frac{8}{\left(\frac{4}{\mathfrak{b}^2(4\tau)}\right)-1}=\frac{8\mathfrak{b}^2(4\tau)}{4-\mathfrak{b}^2(4\tau)},
\end{align*}
by (\ref{eqn:5.1}).  Now using the expression $\mathfrak{b}(4\tau)=2\varphi(-q)/\varphi(q)$ from the proof of Proposition \ref{prop:7}a) and (\ref{eqn:2.18}) gives 
\begin{equation*}
v^{-2}(\tau)+v^2(\tau)-6=\frac{8\left(\frac{4\varphi^2(-q)}{\varphi^2(q)}\right)}{4-\left(\frac{4\varphi^2(-q)}{\varphi^2(q)}\right)} =\frac{8\varphi^2(-q)}{\varphi^2(q)-\varphi^2(-q)} =\frac{8\varphi^2(-q)}{8q\psi^2(q^4)}.
\end{equation*}
Now putting $\varphi(-q)=(q;q^2)_\infty^2(q^2;q^2)_\infty=\frac{(q;q)_\infty^2}{(q^2;q^2)_\infty}$ and $\psi(q)=\frac{(q^2;q^2)_\infty}{(q;q^2)_\infty}=\frac{(q^2;q^2)_\infty^2}{(q;q)_\infty}$ yields
\begin{align*}
v^{-2}(\tau)+v^2(\tau)-6&=\varphi^2(-q)\cdot\left(\frac{1}{q\psi^2(q^4)}\right)\\
&=(q;q^2)_\infty^4(q^2;q^2)_\infty^2\cdot\left(\frac{(q^4;q^8)_\infty^2}{q(q^8;q^8)_\infty^2}\right)\\
&=\left(\frac{(q;q)_\infty^4}{(q^2;q^2)_\infty^2}\right)\cdot\left(\frac{(q^4;q^4)_\infty^2}{q(q^8;q^8)_\infty^4}\right)\\
&=\frac{q^{1/6}(q;q)_\infty^4\cdot q^{1/3}(q^4;q^4)_\infty^2}{q^{1/6}(q^2;q^2)_\infty^2\cdot q^{4/3}(q^8;q^8)_\infty^4}\\
&=\frac{\eta^4(\tau)\eta^2(4\tau)}{\eta^2(2\tau)\eta^4(8\tau)},
\end{align*}
using that $\eta(\tau)= q^{1/24}(q;q)_\infty$.
\end{proof}

\section{The field generated by $v(w/8)$.}

In this section we take $\tau = w/8$, where $-d \equiv 1$ (mod $8$) and
\begin{equation}
w = \frac{a+\sqrt{-d}}{2}, \ \ \textrm{with} \ a^2+d \equiv 0 \ (\textrm{mod} \ 2^5), \ (N(w), f) = 1.
\label{eqn:6.1}
\end{equation}
For this value of $w$,
$$\mathfrak{b}^4(8\tau) = \mathfrak{b}^4(w)$$
is the fourth power of the number
\begin{equation}
\beta = i^{-a} \mathfrak{b}(w)
\label{eqn:6.2}
\end{equation}
from \cite[(10.3), Thms. 10.6, 10.7]{lm}.  We also need the number $\pi$ from \cite[(10.2),(10.9)]{lm}, which is given by 
\begin{align*}
\pi &= i^{\bar c} \frac{\mathfrak{f}_2(w/2)^2}{\mathfrak{f}(w/2)^2} = i^{\bar c} \mathfrak{p}(w),\\
\bar c & \equiv a\left(2-\frac{a^2+d}{16}\right) \ \ (\textrm{mod} \ 4).
\end{align*}
(We have replaced $v$ in the formulas of \cite{lm} by $a$ and $a$ by $\bar c$.) But here the integer $a^2+d$ is divisible by $32$, by (\ref{eqn:6.1}), so $\bar c$ is even.  Replacing $\bar c$ by the integer $c = \bar c/2$, satisfying
$$c \equiv 1-\frac{a^2+d}{32} \ (\textrm{mod} \ 2)$$
yields
\begin{equation}
\pi =  (-1)^{c} \mathfrak{p}(w), \ \ \ w = \frac{a+\sqrt{-d}}{2}.
\label{eqn:6.3}
\end{equation} 
It follows from the results of \cite{lm} that $\xi = \beta/2$ and $\pi$ lie in the ring class field $\Omega_f$ of the quadratic field $K = \mathbb{Q}(\sqrt{-d})$ (where $-d=\mathfrak{d}_K f^2$ and $\mathfrak{d}_K$ is the discriminant of $K/\mathbb{Q}$) and $\xi^4+\pi^4 = 1$.  Furthermore, $\mathbb{Q}(\pi) = \mathbb{Q}(\pi^4) = \Omega_f$.  We also note that $(\xi) = \wp_2'$ and $(\pi) = \wp_2$ in $\Omega_f$, so that $(\xi \pi) = (2)$.
\medskip

From (\ref{eqn:4.3}) and (\ref{eqn:6.3}) we have that
\begin{equation}
(-1)^{c}\frac{2}{\pi} = \frac{1}{v(w/8)} - v(w/8) = \frac{1-v^2(w/8)}{v(w/8)}.
\label{eqn:6.4}
\end{equation}
In particular, $v(w/8)$ satisfies a quadratic equation over $\Omega_f$ and the map $\rho: v(w/8) \rightarrow \frac{-1}{v(w/8)}$ leaves the right side of (\ref{eqn:6.4}) invariant.  On squaring (\ref{eqn:6.4}), we see that $X = v^2(w/8)$ satisfies the equation
\begin{equation}
X^2-(2+\frac{4}{\pi^2})X + 1 = 0,
\label{eqn:6.5}
\end{equation}
and therefore
$$v^2(w/8) = \frac{\pi^2+2 \pm 2 \sqrt{\pi^2+1}}{\pi^2} = \left(\frac{1 \pm \sqrt{1+\pi^2}}{\pi}\right)^2.$$
Hence
\begin{equation}
v(w/8) = \pm \frac{1 \pm \sqrt{1+\pi^2}}{\pi}.
\label{eqn:6.6}
\end{equation}
It follows from these expressions that $\Omega_f(v(w/8)) = \Omega_f(v^2(w/8)) = \Omega_f(\sqrt{1+\pi^2})$.
\medskip

Let $-d \equiv 1$ (mod $8$) and set $-d = \mathfrak{d}_K f^2$, where $\mathfrak{d}_K$ is the discriminant of the field
$K = \mathbb{Q}(\sqrt{-d})$.  Further, let $2 \cong \wp_2 \wp_2'$ in the ring of integers $R_K$ of $K$.  We denote by $\Sigma_\mathfrak{f}$ the ray class field of conductor $\mathfrak{f}$ over $K$ and $\Omega_f$ the ring class field of conductor $f$ over $K$. \bigskip

We now prove the following.

\newtheorem{thm}{Theorem}

\begin{thm} If
$$w = \frac{a+\sqrt{-d}}{2}, \ \ \textrm{with} \ a^2+d \equiv 0 \ (\textrm{mod} \ 2^5),$$
and $\wp_2 = (2,w)$ in $R_K$, then the field $\mathbb{Q}(v(w/8)) = \mathbb{Q}(\sqrt{1+\pi^2})$ coincides with the field $\Sigma_{\wp_2'^3} \Omega_f$.  The units $v(w/8)$ and $v^2(w/8)$ have degree $4h(-d)$ over $\mathbb{Q}$.
\label{thm:1}
\end{thm}

\begin{proof}
Let $\Lambda = \mathbb{Q}(\sqrt{1+\pi^2})$.  It is clear that $\Lambda$ contains the ring class field $\Omega_f$, since $\mathbb{Q}(\pi^4) = \Omega_f$.  We use the fact that $1+\pi^2 \cong \wp_2'$ from \cite[Lemma 5]{m}.  From this fact it is clear that $1+\pi^2$ is not a square in $\Omega_f$, since $\wp_2'$ is unramified in $\Omega_f/K$.  Hence, $[\Lambda: \Omega_f] =  2$.  Further, the prime divisors $\mathfrak{q}$ of $\wp_2'$ in $\Omega_f$ are certainly ramified in $\Lambda$.  Equation (\ref{eqn:6.5}) implies that $x = v^2(w/8)$ satisfies $(x-1)^2/(4x) = 1/\pi^2$, and therefore $\mathbb{Q}(v^2(w/8)) = \mathbb{Q}(\sqrt{1+\pi^2})$.  This implies that $[\mathbb{Q}(v^2(w/8)):\mathbb{Q}] = 4h(-d)$, since
$$[\Lambda: \mathbb{Q}] = [\Lambda: \Omega_f] [\Omega_f: K] [K: \mathbb{Q}] = 4h(-d).$$
Since $v^2(\tau)$ is a modular function for $\Gamma_1(8)$ (\cite[p.154]{d}), it follows from Schertz \cite[Thm. 5.1.2]{sch} that $v^2(w/8) \in \Sigma_{8f}$, the ray class field of conductor $8f$ over $K$.  More precisely, $v^2(w/8) \in L_{\mathcal{O},8}$, where $L_{\mathcal{O},8} = \Sigma_8 \Omega_f$ is an {\it extended ring class field} corresponding to the order $\mathcal{O} = \textsf{R}_{-d}$.  See \cite[p. 315]{co}.  Thus, $\Lambda \subset L_{\mathcal{O},8}$ is an abelian extension of $K$, whose conductor $\mathfrak{f}$ divides $8f$ in $K$.  The discriminant of the polynomial $X^2-(1+\pi^2)$ is of course $4(1+\pi^2) \cong \wp_2^2 \wp_2'^3$.  Since the ramification index of each $\mathfrak{q} \mid \wp_2'$ is $e_\mathfrak{q} = 2$ in $\Lambda/\Omega_f$, Dedekind's discriminant theorem says that at least $\wp_2'^2$ divides the discriminant $\mathfrak{d} = \mathfrak{d}_{\Lambda/\Omega_f}$, and since the power of $\mathfrak{q}$ in $\mathfrak{d}$ is odd and at most $3$ ($\Omega_f/K$ is unramified over $2$), it follows that $\wp_2'^3$ exactly divides $\mathfrak{d}$.  We claim now that $\wp_2$ is unramified in $\Lambda$. \medskip

From above $x = v^2(w/8)$ satisfies $(x-1)^2 - \frac{4}{\pi^2}x = 0$.  Thus $x_1 = x-1$ satisfies $h(x_1) = 0$, with
$$h(X) = X^2 - \frac{4}{\pi^2}(X+1),  \ \ \textrm{disc}(h(X)) = \frac{16}{\pi^4} + 4\frac{4}{\pi^2},$$
where the ideal $(\frac{16}{\pi^4}) = \left(\frac{2}{\pi}\right)^4 = (\xi)^4 =\wp_2'^4$ is not divisible by $\wp_2$.  This shows that $\textrm{disc}(h(X))$ is not divisible by $\wp_2$ and therefore that $\wp_2$ is unramified in $\mathbb{Q}(v^2(w/8))$.  Thus $\mathfrak{d} = \wp_2'^3$. \medskip

Now $[\Sigma_8: \Sigma_1] = \frac{1}{2} \phi_K(\wp_2^3 \wp_2'^3) = 8,$ where $\phi_K$ is the Euler function for the quadratic field $K$, and $\mathbb{Q}(\zeta_8) \subset \Sigma_8$.  Since the prime divisors of $2$ do not ramify in $\Omega_f$, we have that $\Omega_f \cap \Sigma_8 = \Sigma_1$ and therefore
$$[L_{\mathcal{O},8}:\Omega_f] = [\Sigma_8 \Omega_f: \Omega_f] = [\Sigma_8: \Sigma_1] = 8,$$
from which we obtain
$$\textrm{Gal}(\Sigma_8 \Omega_f/ \Omega_f) \cong \textrm{Gal}(\Sigma_8/ \Sigma_1).$$
By this isomorphism the intermediate fields $L\Omega_f$ of $\Sigma_8 \Omega_f/ \Omega_f$ are in $1-1$ correspondence with the intermediate fields $L$ of $\Sigma_8/ \Sigma_1$.  \medskip

Now the ray class field $\Sigma_{\wp_2^2\wp_2'^3}$ has degree $4$ over the Hilbert class field $\Sigma_1$, and two of its quadratic subfields are $\Sigma_{\wp_2'^3}$ and $\Sigma_{\wp_2^2 \wp_2'^2} = \Sigma_4 = \Sigma_1(i)$.  It follows that $\textrm{Gal}(\Sigma_{\wp_2^2\wp_2'^3}/\Sigma_1) \cong \mathbb{Z}_2 \times \mathbb{Z}_2$ and the third quadratic subfield has conductor equal to $\mathfrak{f}' = \wp_2^2\wp_2'^3$ over $K$.  The other quadratic intermediate fields of $\Sigma_8/\Sigma_1$ are $\Sigma_1(\sqrt{2})$ and $\Sigma_1(\sqrt{-2})$,  both of which have conductor $(8) = \wp_2^3 \wp_2'^3$ over $K$, the field $\Sigma_{\wp_2^3}$, and a field whose conductor over $K$ is $\wp_2'^2\wp_2^3$.  Hence, $L = \Sigma_{\wp_2'^3}$ is the only quadratic intermediate field whose conductor is not divisible by $\wp_2$.  This proves that $\mathbb{Q}(v^2(w/8)) = \Sigma_{\wp_2'^3} \Omega_f$ is a quadratic extension of $\Omega_f$ and (\ref{eqn:6.6}) shows that $\mathbb{Q}(v(w/8)) = \mathbb{Q}(v^2(w/8)) = \Sigma_{\wp_2'^3} \Omega_f$.
\end{proof}

\noindent {\bf Corollary.} {\it The field $\mathbb{Q}(v(w/8)) = \Sigma_{\wp_2'^3}\Omega_f$ is the inertia field for the prime ideal $\wp_2$ in the extension $L_{\mathcal{O},8}/K = \Sigma_8 \Omega_f/K$.}

\begin{proof}
The above proof implies that $\textrm{Gal}(\Sigma_8 \Omega_f/\Omega_f) \cong \mathbb{Z}_2 \times \mathbb{Z}_2 \times \mathbb{Z}_2$, since there are $7$ quadratic intermediate fields.  Any subfield containing $\Omega_f$ which properly contains $\Sigma_{\wp_2'^3}$ must also contain another quadratic subfield, in which $\wp_2$ must ramify.
\end{proof}

The fact that $v^2(w/8) \in L_{\mathcal{O},8}$ in the above proof is derived using Shimura's Reciprocity Law.  We can give a more elementary proof of this fact by showing that $\sqrt{1+\pi^2} \in L_{\mathcal{O},8}$, as follows.  We focus on the elliptic curve
$$E_1(\alpha): \ Y^2 + XY+\frac{1}{\alpha^4}Y = X^3 + \frac{1}{\alpha^4}X^2,$$
which is the Tate normal form for a point of order $4$, with
$$\alpha^4 = \alpha(w)^4 = -\left(\frac{\eta(w/4)}{\eta(w)}\right)^8,$$
as in (\ref{eqn:5.2}).  From \cite[(2.10), Prop. 3.2, p. 1970]{lm}, the curve $E_1 = E_1(\alpha)$ has complex multiplication by the order $\mathcal{O} = \textsf{R}_{-d}$ of discriminant $-d$ in $K$.  Now, with $\beta = i^{-a} \mathfrak{b}(w)$ as in (\ref{eqn:6.2}),
$$\frac{1}{\alpha^4} = \frac{\beta^4-16}{16\beta^4} = \frac{1}{16} - \frac{1}{\beta^4},$$
and Lynch \cite{ly} has given explicit expressions for the points of order $8$ on $E_1$ in terms of $\beta$.  Lynch \cite[Prop. 3.3.1, p. 38]{ly} defines the following expressions:
\begin{align*}
b_1 & = \frac{\beta \sqrt{2}+(\beta^2+4)^{1/2}+(\beta^2-4)^{1/2}}{2\beta \sqrt{2}},\\
b_2 & = \frac{\beta \sqrt{2}+(\beta^2+4)^{1/2}-(\beta^2-4)^{1/2}}{2\beta \sqrt{2}},\\
b_3 & = \frac{\beta \sqrt{2}-(\beta^2+4)^{1/2}+(\beta^2-4)^{1/2}}{2\beta \sqrt{2}},\\
b_4 & = \frac{\beta \sqrt{2}-(\beta^2+4)^{1/2}-(\beta^2-4)^{1/2}}{2\beta \sqrt{2}}.
\end{align*}
With these expressions, Lynch shows \cite[Thm. 3.3.1, p. 41]{ly} that the points
$$(X,Y) = P_1 = (b_1b_3b_4, -b_1b_3^3 b_4) \ \ \textrm{and} \ \ P_2 = (b_2 b_3 b_4, -b_2b_3b_4^3)$$
are points of order $8$ on $E_1(\alpha)$.  By \cite[Satz 2]{fr} or \cite[Prop. 6.4]{lm} the corresponding Weber functions satisfy
$$\frac{g_2 g_3}{\Delta} \left(X(P_i) + \frac{4b+1}{12}\right) \in \Sigma_8\Omega_f, \ \ \ b = \frac{1}{\alpha^4}.$$
(See \cite[(6.1)]{lm}. The expression inside the parentheses arises from putting the curve $E_1(\alpha)$ in standard Weierstrass form.)  As in \cite[p. 1976]{lm}, $b, g_2, g_3, \Delta \in \Omega_f$, so that $X(P_i) = b_i b_3 b_4 \in L_{\mathcal{O},8}$ for $i = 1,2$.  This implies that
\begin{align*}
(b_1+b_2)b_3b_4 &= \left(\frac{\sqrt{2}\beta+(\beta^2+4)^{1/2}}{\sqrt{2} \beta}\right) \left(\frac{\beta^2+4-\sqrt{2} \beta (\beta^2+4)^{1/2}}{4\beta^2}\right)\\
& = \frac{4-\beta^2}{4\sqrt{2} \beta^3} (\beta^2+4)^{1/2}
\end{align*}
lies in $L_{\mathcal{O},8}$.  But  we know that $4-\beta^2 \neq 0$.  In addition, $\sqrt{2} \in \mathbb{Q}(\zeta_8) \subset \Sigma_8$ and $\beta \in \Omega_f$, so that $(\beta^2+4)^{1/2} = 2 \sqrt{\xi^2+1} \in L_{\mathcal{O},8}$, with $\xi = \beta/2$.  Now $\pi$ and $\xi$ are conjugate over $\mathbb{Q}$, hence $\pm \sqrt{1+\pi^2}$ is conjugate to $\sqrt{1+\xi^2}$ over $\mathbb{Q}$.  Since $\Sigma_8 \Omega_f$ is normal over $\mathbb{Q}$, this implies that $\sqrt{1+\pi^2} \in L_{\mathcal{O},8}$, which proves the assertion.  \medskip

\begin{prop} Assume $c$ in (\ref{eqn:6.3}) is odd.  The map $A(x) = \frac{\sigma x +1}{x-\sigma}$ (see (\ref{eqn:3.3})) fixes the set of conjugates of $v(w/8)$.  If $f_d(x)$ is the minimal polynomial of $v(w/8)$ over $\mathbb{Q}$, then
$$(x-\sigma)^{4h(-d)} f_d(A(x)) = 2^{3h(-d)} \sigma^{2h(-d)}f_d(x).$$
\label{prop:10}
\end{prop}

\begin{proof}
Note that (\ref{eqn:6.4}) implies that the minimal polynomial of $v(w/8)$ is
\begin{equation}
f_d(x) = 2^{-h(-d)}(x^2-1)^{2h(-d)}b_d\left((-1)^{c} \frac{2x}{1-x^2}\right),
\label{eqn:6.7}
\end{equation}
where $b_d(x)$ is the minimal polynomial of $\pi$.  Note that the degree of $b_d(x)$ is $2h(-d)$ and the constant term of $b_d(x)$ is
$$N_{\Omega_f/\mathbb{Q}}(\pi)= N_{\Omega_f/\mathbb{Q}}(\wp_2)=N_{K/\mathbb{Q}}(\wp_2^{h(-d)})=2^{h(-d)}$$
from \cite{lm}.  Thus, $\textrm{deg}(f_d(x)) = 4h(-d)$, which implies by Theorem \ref{thm:1} that $f_d(x)$ is irreducible.  \medskip

We use (\ref{eqn:6.7}) to prove the proposition, as follows.  Setting $h = h(-d)$ and assuming $c$ is odd, we have that
\begin{align*}
(x-\sigma)^{4h}f_d(A(x)) & = 2^{-h}(x-\sigma)^{4h}(A(x)^2-1)^{2h} b_d\left(\frac{2A(x)}{A(x)^2-1}\right)\\
& =2^{-h}(x-\sigma)^{4h} \left(\frac{-2 \sigma(x^2-2x-1)}{(x-\sigma)^2}\right)^{2h} b_d\left(-\frac{x^2 + 2x - 1}{x^2 - 2x - 1}\right)\\
& = 2^{h} \sigma^{2h} (x^2 -2x-1)^{2h} b_d\left(\frac{P(x)+1}{P(x)-1}\right),
\end{align*}
where
$$P(x) = \frac{2x}{x^2-1}$$
and
$$\frac{P(x)+1}{P(x)-1} = -\frac{x^2 + 2x - 1}{x^2 - 2x - 1} = R(x).$$
We also know from \cite{lm} that the map $x \rightarrow \frac{x+1}{x-1}$ permutes the roots of $b_d(x)$ and
$$(x-1)^{2h}b_d\left(\frac{x+1}{x-1}\right) = 2^h b_d(x).$$
This gives that
$$b_d\left(\frac{P(x)+1}{P(x)-1}\right) = (P(x)-1)^{-2h} 2^h b_d(P(x))$$
and therefore that
\begin{align*}
(x-\sigma)^{4h}f_d(A(x)) & = 2^{h} \sigma^{2h} (x^2 -2x-1)^{2h} (P(x)-1)^{-2h} 2^h b_d(P(x))\\
& = 2^{2h} \sigma^{2h} (x^2 -2x-1)^{2h} \left(\frac{x^2-1}{x^2-2x-1}\right)^{2h} b_d(P(x))\\
& = 2^{3h} \sigma^{2h} 2^{-h} (x^2-1)^{2h} b_d(P(x))\\
& = 2^{3h} \sigma^{2h} f_d(x).
\end{align*}
\end{proof} 

We also check that
\begin{align*}
x^{4h} f_d\left(\frac{-1}{x}\right) & = 2^{-h} x^{4h} \left(\frac{1}{x^2}-1\right)^{2h(-d)}b_d(P(-1/x))\\
& = 2^{-h} (x^2-1)^{2h} b_d(P(x)) = f_d(x).
\end{align*}
We conclude the following.  Recall the definition of $\bar A(x)$ from (\ref{eqn:3.3}).

\begin{prop} If $c$ is odd, the mappings in the group
$$\tilde H_{1} = \{x, A(x), \bar A(x), -1/x\}$$
permute the roots of $f_d(x)$.
\label{prop:11}
\end{prop}

Now let $c$ be even, $\delta=1+\sqrt{2}$, and $\displaystyle B(x)=\frac{\delta x+1}{x-\delta}=\frac{x+\sigma}{\sigma x-1} = -\bar A(-x)$. Then we have
\begin{align*}
(x-\delta)^{4h}f_d(B(x))&=2^{-h}(x-\delta)^{4h}(B^2(x)-1)^{2h} b_d\left(\frac{2B(x)}{1-B^2(x)}\right)\\
&=2^{-h}(x-\delta)^{4h} \left(\frac{2\delta(x^2+2x-1)}{(x-\delta)^2}\right)^{2h} b_d\left(-\frac{x^2 - 2x - 1}{x^2 + 2x - 1}\right)\\
&=2^{h}\delta^{2h}(x^2+2x-1)^{2h}\ b_d\left(\frac{\frac{2x}{1-x^2}+1}{\frac{2x}{1-x^2}-1}\right)\\
&=2^{h}\delta^{2h}(x^2+2x-1)^{2h}\cdot2^h\bigg(\frac{2x}{1-x^2}-1\bigg)^{-2h}b_d\left(\frac{2x}{1-x^2}\right)\\
&=2^{2h}\delta^{2h}(x^2+2x-1)^{2h}\cdot\left(\frac{1-x^2}{x^2+2x-1}\right)^{2h} b_d\left(\frac{2x}{1-x^2}\right)\\
&=2^{2h}\delta^{2h}\cdot(x^2-1)^{2h}b_d\left(\frac{2x}{1-x^2}\right)\\
&=2^{2h}\delta^{2h}\cdot2^hf_d(x)\\
&=2^{3h}\delta^{2h}f_d(x).
\end{align*}

Now set $\displaystyle \bar B(x) = B(-1/x) = \frac{-\sigma x+1}{x+\sigma} = -A(-x)$. \bigskip

\begin{prop} If $c$ is even, the mappings in the group
$$\tilde H_{0} = \{x, B(x), \bar B(x), -1/x\}$$
permute the roots of $f_d(x)$.
\label{prop:12}
\end{prop}

\section{The diophantine equation.}

From (\ref{eqn:3.4}) we know that $(X,Y) = (v(w/8),v(-1/w))$ is a solution of the diophantine equation
$$\mathcal{C}_2: \ X^2 + Y^2 = \sigma^2(1+X^2 Y^2), \ \ \  \sigma = -1+\sqrt{2}.$$
This seems to be an analogue of the equation $\mathcal{C}_5$ in \cite{m2}.  Set
$$F_2(X,Y) = X^2 + Y^2 - \sigma^2(1+X^2 Y^2).$$
Then
$$(\sigma Y+1)^2 F_2(X,\bar A(Y)) = 4\sqrt{2} \sigma^2(X^2 Y+X^2 +Y^2-Y) = 4\sqrt{2} \sigma^2 f(X,Y).$$
Since
$$\bar A(x) = \frac{-x+\sigma}{\sigma x+1} = \frac{-\delta x+1}{x+\delta}, \ \ \delta = \frac{1}{\sigma} = 1+ \sqrt{2},$$
the linear fractional map $\bar A(x)$ is the analogue of the map $T(x)$ in \cite[p. 1199]{m2}.  Considering Thm. 5.1 in \cite[p. 1205]{m2} suggests the following conjecture. \medskip

\noindent {\bf Conjecture.} {\it Assume $c$ is odd.  If $\displaystyle \tau_2 = \left(\frac{\Sigma_{\wp_2'^3}\Omega_f/K}{\wp_2}\right)$, then}
$$-v(-1/w) = \bar A(v(w/8)^{\tau_2}) = \frac{-v(w/8)^{\tau_2}+\sigma}{\sigma v(w/8)^{\tau_2}+1},$$
{\it where $w$ is given by (\ref{eqn:6.1}).}
\bigskip

To prove this conjecture, we first appeal to Proposition \ref{prop:6}, which implies that
$$\mathfrak{p}(2\tau) = \frac{1 \pm \sqrt{1-\mathfrak{p}^4(\tau)}}{\mathfrak{p}^2(\tau)}.$$
Setting $\tau=w$, (\ref{eqn:6.3}) gives that
$$\mathfrak{p}(2w) = \frac{1 \pm \sqrt{1-\pi^4}}{\pi^2} = \frac{1 \pm \xi^2}{\pi^2}.$$
Note that
$$\frac{1 + \xi^2}{\pi^2} \frac{1 - \xi^2}{\pi^2} = \frac{1 - \xi^4}{\pi^4} = 1$$
and $\displaystyle \frac{1 - \xi^2}{\pi^2} = -\pi^{\tau_2}$ from \cite[p. 333]{m}. Thus, $\displaystyle \frac{1 + \xi^2}{\pi^2} = -\pi^{-\tau_2}$. \bigskip

\begin{thm} If $w$ is given by (\ref{eqn:6.1}) we have
$$\mathfrak{p}(2w) = \frac{1 + \xi^2}{\pi^2} = \frac{-1}{\pi^{\tau_2}}.$$
\label{thm:2}
\end{thm}

\begin{proof} We use an argument from \cite[Section 10]{lm}.  With the number $\beta = i^{-a}\mathfrak{b}(w)$ from (\ref{eqn:6.2}) we have \cite[eq. (8.0), p. 1980]{lm}
$$j(w) = \frac{(\beta^8 + 224\beta^4 + 256)^3}{\beta^4(\beta^4-16)^4}.$$
(See the proof of Proposition \ref{prop:9}.)  Furthermore, the roots of the equation
$$0 = (X-16)^3 - j(w)X = (X-16)^3 -  \frac{(\beta^8 + 224\beta^4 + 256)^3}{\beta^4(\beta^4-16)^4}X$$
are, on the one hand, given by the values
$$X = \mathfrak{f}^{24}(w), \ \ -\mathfrak{f}_1^{24}(w), \ \ -\mathfrak{f}_2^{24}(w);$$
(see \cite[p. 233, Th. 12.17]{co}) and on the other, are equal to the expressions
$$X = -\frac{(\beta^2-4)^4}{\beta^2(\beta^2+4)^2}, \ \ \frac{(\beta^2+4)^4}{\beta^2(\beta^2-4)^2}, \ \ -\frac{2^{12}\beta^4}{(\beta^4-16)^2}.$$
See \cite[p. 2000]{lm}.  From \cite[p. 2000]{lm} we also have (since our value $w$ satisfies the conditions for $w$ in \cite[Prop. 3.1]{lm})
\begin{equation}
\mathfrak{f}_2^{24}(w) = - \frac{(\beta^2+4)^4}{\beta^2(\beta^2-4)^2},
\label{eqn:7.1}
\end{equation}
since $\mathfrak{f}_2^{24}(w)$ must be a unit (from the results of \cite{yz}).  There are two cases to consider. \medskip

\noindent {\it Case 1.} First assume that
\begin{align}
\label{eqn:7.2} \mathfrak{f}^{24}(w) & = -\frac{(\beta^2 - 4)^4}{\beta^2(\beta^2 + 4)^2},\\
\notag \mathfrak{f}_1^{24}(w) & = 2^{12} \frac{2^{12}\beta^4}{(\beta^4 - 16)^2}.
\end{align}
In this case, (\ref{eqn:7.1}) and (\ref{eqn:7.2}) give the following formula:
$$\mathfrak{p}^{12}(2w) = \frac{\mathfrak{f}_2(w)^{24}}{\mathfrak{f}(w)^{24}} = \frac{(\beta^2 + 4)^6}{(\beta - 2)^6 (\beta + 2)^6}.$$
Now we use the following ideal factorizations in the ring class field $\Omega_f$:
\begin{equation}
(\beta^2+4) = \wp_2^3 \wp_2'^2, \ \ (\beta-2) = \wp_2^2 \wp_2', \ \ (\beta+2) = \wp_2^3 \wp_2'.
\label{eqn:7.3}
\end{equation}
See \cite[Lemma 4]{m}.  These factorizations imply that
$$\mathfrak{p}^{12}(2w) \cong \left(\frac{\wp_2^3 \wp_2'^2}{\wp_2^5 \wp_2'^2}\right)^6 = \frac{1}{\wp_2^{12}} \ \ \textrm{in} \ \Omega_f,$$
which implies that
\begin{equation}
\mathfrak{p}(2w) \cong \frac{1}{\wp_2}.
\label{eqn:7.4}
\end{equation}
By the remarks preceding the statement of the theorem, this shows that $\mathfrak{p}(2w)$ is not an algebraic integer, giving that
$\displaystyle \mathfrak{p}(2w) = \frac{1 + \xi^2}{\pi^2} = -\pi^{-\tau_2}$.
\medskip

\noindent {\it Case 2.} The alternative to (\ref{eqn:7.2}) is
\begin{align}
\label{eqn:7.5} \mathfrak{f}^{24}(w) & = - \frac{2^{12}\beta^4}{(\beta^4 - 16)^2},\\
\notag \mathfrak{f}_1^{24}(w) & = \frac{(\beta^2 - 4)^4}{\beta^2(\beta^2 + 4)^2}.
\end{align}
In this case we have
$$\mathfrak{p}^{12}(2w) = \frac{\mathfrak{f}_2(w)^{24}}{\mathfrak{f}(w)^{24}} = \left(\frac{\beta^2 + 4}{2^{2} \beta}\right)^6 \cong \left(\frac{\wp_2^3 \wp_2'^2}{\wp_2^2 \wp_2'^2 \wp_2 \wp_2'^2}\right)^6 = \frac{1}{\wp_2'^{12}},$$
giving that $\mathfrak{p}(2w) \cong \frac{1}{\wp_2'}$.  However, this is impossible, since the above remarks show that the only prime divisors occuring in the factorization of $\mathfrak{p}(2w)$ are prime divisors of $\wp_2$.  This shows that Case 2 is impossible, and Case 1 proves the formula of the theorem.
\end{proof}

Now we set
\begin{equation}
\eta = v(w/8), \ \ \lambda = -v(-1/w), \ \ \nu = v(w/4).
\label{eqn:7.6}
\end{equation}
We first show $\lambda$ is a root of the minimal polynomial $f_d(x)$ of $v(w/8)$ ($c$ odd).  We have from Proposition \ref{prop:3} that
$$\frac{2\lambda}{\lambda^2-1} = \frac{-2 \bar A(\nu)}{\bar A^2(\nu)-1} = \frac{\nu^2+2\nu-1}{\nu^2-2\nu-1}.$$
Proposition \ref{prop:5} and Theorem \ref{thm:2} give further that
\begin{equation}
\frac{2\lambda}{\lambda^2-1} = \frac{\nu-\frac{1}{\nu}+2}{\nu-\frac{1}{\nu}-2} = \frac{\frac{-2}{\mathfrak{p}(2w)}+2}{\frac{-2}{\mathfrak{p}(2w)}-2} = \frac{\pi^{\tau_2}+1}{\pi^{\tau_2}-1}.
\label{eqn:7.7}
\end{equation}
Since $\frac{\pi^{\tau_2}+1}{\pi^{\tau_2}-1}$ is a root of $b_d(x)$, we have from (\ref{eqn:6.7}) that
$$f_d(\lambda) = 2^{-h(-d)}(\lambda^2-1)^{2h(-d)}b_d\left(\frac{2\lambda}{\lambda^2-1}\right) = 0.$$
Hence, $\lambda = -v(-1/w)$ is a conjugate of $v(w/8)$.

\begin{thm} If $c$ is odd, we have the formula
$$\lambda = -v(-1/w) = \bar A(v(w/8)^{\tau_2}) = \frac{-v(w/8)^{\tau_2}+\sigma}{\sigma v(w/8)^{\tau_2}+1}, \ \ \sigma = -1+\sqrt{2},$$
where $w$ is given by (\ref{eqn:6.1}).
\label{thm:3}
\end{thm}

\begin{proof}
We will prove that $\bar A(\lambda) = v(w/8)^{\tau_2} = \eta^{\tau_2}$ by showing that
$$\bar A(\lambda) - \eta^2 \equiv 0 \ (\textrm{mod} \ \wp_2).$$
We have $\eta^2 + \lambda^2 = \sigma^2(1+\eta^2 \lambda^2)$, which implies that
\begin{align*}
\bar A(\lambda) - \eta^2 & = \frac{-\lambda+\sigma}{\sigma \lambda+1} - \frac{-\lambda^2+\sigma^2}{1-\sigma^2 \lambda^2} = \frac{-\lambda+\sigma}{\sigma \lambda+1} + \frac{\sigma^2-\lambda^2}{\sigma^2 \lambda^2-1}\\
& = \frac{(-\lambda+\sigma)(\sigma \lambda - 1) + \sigma^2-\lambda^2}{\sigma^2 \lambda^2-1}\\
& = \frac{-(\sigma+1)\lambda^2+(\sigma^2+1)\lambda+\sigma^2-\sigma}{\sigma^2 \lambda^2-1}\\
& = \frac{-\sqrt{2} \lambda^2+(4-2\sqrt{2})\lambda+4-3\sqrt{2}}{(\sigma \lambda+1)(\sigma \lambda-1)}\\
& = \frac{-\sqrt{2} (\lambda-\sigma)^2}{\sigma^2(\lambda-\bar \sigma)(\lambda+\bar \sigma)}.
\end{align*}
We multiply the last expression by
$$A(\lambda) - \frac {1}{\eta^2} = \frac{(-4+3\sqrt{2})(\lambda-\bar \sigma)^2}{\lambda^2-\sigma^2}=\frac{\sqrt{2}\sigma^2(\lambda-\bar \sigma)^2}{\lambda^2-\sigma^2},$$
which is obtained from the last calculation by fixing $\lambda$ and mapping $\sqrt{2}$ to $-\sqrt{2}$.  This yields the formula
\begin{equation}
\left(\bar A(\lambda) - \eta^2\right) \left(A(\lambda) - \frac {1}{\eta^2}\right) = \frac{-2(\lambda-\sigma)(\lambda-\bar \sigma)}{(\lambda+\sigma)(\lambda+ \bar \sigma)} = -2\frac{\lambda^2+2\lambda-1}{\lambda^2-2\lambda-1}.
\label{eqn:7.8}
\end{equation}
Now
\begin{equation}
\frac{\lambda^2+2\lambda-1}{\lambda^2-2\lambda-1} = \frac{1+\frac{2\lambda}{\lambda^2-1}}{1-\frac{2\lambda}{\lambda^2-1}},
\label{eqn:7.9}
\end{equation}
where
\begin{equation*}
\frac{2\lambda}{\lambda^2-1} = \frac{\pi^{\tau_2}+1}{\pi^{\tau_2}-1}
\end{equation*}
from (\ref{eqn:7.7}).  It follows from (\ref{eqn:7.9}) that
$$\frac{\lambda^2+2\lambda-1}{\lambda^2-2\lambda-1} = \frac{1+\frac{\pi^{\tau_2}+1}{\pi^{\tau_2}-1}}{1-\frac{\pi^{\tau_2}+1}{\pi^{\tau_2}-1}} = -\pi^{\tau_2}.$$
Thus, (\ref{eqn:7.8}) becomes
$$\left(\bar A(\lambda) - \eta^2\right) \left(A(\lambda) - \frac {1}{\eta^2}\right) = 2\pi^{\tau_2}$$
and therefore $(\pi^{\tau_2}) = (\pi) = \wp_2$ yields that
$$\left(\bar A(\lambda) - \eta^2\right) \left(A(\lambda) - \frac {1}{\eta^2}\right) \equiv 0 \ (\textrm{mod} \ \wp_2^2).$$
It follows that
\begin{equation}
\bar A(\lambda) \equiv \eta^2 \ \ \textrm{or} \ \ A(\lambda) \equiv \frac {1}{\eta^2} \ (\textrm{mod} \ \mathfrak{q}),
\label{eqn:7.10}
\end{equation}
for each prime divisor $\mathfrak{q}$ of $\wp_2$ in $F_1 = \mathbb{Q}(\eta)$.  But $A(\lambda) = -1/\bar A(\lambda)$ and $\eta$ are units, so the second congruence in (\ref{eqn:7.10}) implies the first.  This proves that
\begin{equation}
\bar A(\lambda) \equiv \eta^2 \ (\textrm{mod} \ \wp_2)
\label{eqn:7.11}
\end{equation}
in $F_1$.  Note that $\bar A(\lambda)$ and $\lambda = -v(-1/w)$ are roots of $f_d(x)$ (Proposition \ref{prop:11}),  so $F_2 = \mathbb{Q}(\lambda)$ is isomorphic to $F_1 = \mathbb{Q}(\eta) = \mathbb{Q}(v(w/8))$.  However, by (\ref{eqn:3.4}),
$$\lambda^2 = v^2(-1/w) = \frac{-v(w/8)^2+\sigma^2}{1-\sigma^2 v(w/8)^2}$$
does not lie in $F_1$, since $\sqrt{2} \notin F_1$ (otherwise $\wp_2$ would be ramified in $F_1$;  note that $v(w/8)$ is not a fourth root of unity, so the determinant of the linear fractional transformation in $\sigma^2$ is nonzero).  It follows that from Theorem \ref{thm:1} that
$$F_2 = \mathbb{Q}(\lambda) = \Sigma_{\wp_2^3} \Omega_f.$$
The same argument now shows that $\bar A(\lambda) = \frac{-\lambda+\sigma}{\sigma \lambda+1} \notin F_2$, so $\bar A(\lambda) \in F_1$.  Therefore, $\psi: \eta \rightarrow \bar A(\lambda)$ is an automorphism of $F_1$, and since $\wp_2$ is not ramified in $F_1$ but $\wp_2'$ is, it follows that $\psi$ fixes $\wp_2$, implying that it fixes the field $K$. \medskip

Recalling the rational function $j_2(x)$ from (\ref{eqn:5.6}), a computation on Maple shows that
$$j_2\left(\left(\frac{1-\nu}{1+\nu}\right)^2\right) = j_2(\nu^2) = j_2(v^2(w/4)) = j(w/4),$$
by (\ref{eqn:5.7}).  Now Proposition \ref{prop:3} and the fact that $\bar A(x)$ has order $2$ imply that
$v(w/4) = \bar A(v(-1/w))$ and
\begin{align}
\notag \frac{1-v(w/4)}{1+v(w/4)} & = \frac{1-\bar A(v(-1/w))}{1+\bar A(v(-1/w))}\\
\notag & = \frac{v(-1/w)+\sigma}{-\sigma v(-1/w)+1}\\
\label{eqn:7.12} & = \bar A(-v(-1/w)) = \bar A(\lambda).
\end{align}
This implies that
$$j_2(\bar A(\lambda)^2) = j_2\left(\left(\frac{1-\nu}{1+\nu}\right)^2\right) = j(w/4).$$
On the other hand, equation (\ref{eqn:5.7}) gives
$$j(w/8)^\psi = j_2(\eta^{2\psi}) = j_2(\bar A(\lambda)^2) = j(w/4) = j(w/8)^{\tau_2}.$$
Hence $\psi|_{\Omega_f} = \tau_2|_{\Omega_f}$.  It follows that $\psi = \tau_2$ or $\psi = \rho \tau_2$, where $\rho: \eta \rightarrow -1/\eta$ is the nontrivial automorphism of $F_1/\Omega_f$.  Now, if $\psi = \rho \tau_2$, then by (\ref{eqn:7.11})
$$\eta^\psi = \bar A(\lambda) \equiv \eta^2 \ (\textrm{mod} \ \wp_2)$$
and $\eta^{\tau_2} \equiv \eta^2$ (mod $\wp_2$) imply that
$$\eta^2 \equiv \eta^{\rho \tau_2} = \frac{-1}{\eta^{\tau_2}} \equiv \frac{1}{\eta^2} \ (\textrm{mod} \ \wp_2).$$
It follows from this congruence that $\eta^4+1 \equiv (\eta+1)^4 \equiv 0$ mod $\wp_2$ and hence $\eta \equiv 1$ (mod $\wp_2$), since $\wp_2$ is unramified in $F_1/K$.  This implies in turn that $z = \eta-\eta^{-1} \equiv 0$ (mod $\wp_2$).  But this contradicts (\ref{eqn:4.3}) (with $\tau = w/8$) and (\ref{eqn:6.3}), according to which $z = 2/\pi$ is relatively prime to $\wp_2$.  Hence, $\psi = \tau_2$ must be the Artin symbol for $\wp_2$ in $F_1/K$.  This completes the proof.
\end{proof}

\noindent {\bf Corollary.} {\it Assume $c$ is odd.  If $\displaystyle \tau_2 = \left(\frac{\Sigma_{\wp_2'^3}\Omega_f/K}{\wp_2}\right)$, then
$$v(w/8)^{\tau_2} = \frac{1-v(w/4)}{1+v(w/4)}$$
and}
$$f(v(w/8),v(w/8)^{\tau_2}) = 0.$$

\begin{proof} The first formula is immediate from $\eta^\psi = \eta^{\tau_2} = \bar A(\lambda)$ and (\ref{eqn:7.12}).  The second follows from Proposition \ref{prop:1} and
$$f(v(w/8),v(w/4)) = 0 = f\left(v(w/8),\frac{1-v(w/4)}{1+v(w/4)}\right),$$
since
$$f\left(x, \frac{1-y}{1+y}\right) = \frac{2f(x,y)}{(1+y)^2}.$$
\end{proof}

\begin{thm} If $c$ is even, then
$$v(w/8)^{\tau_2}=\frac{v(w/4)-1}{v(w/4)+1}$$ and
$$v(-1/w)=B(v(w/8)^{\tau_2})=\frac{v(w/8)^{\tau_2}+\sigma}{\sigma v(w/8)^{\tau_2}-1}.$$
\label{thm:4}
\end{thm}
\begin{proof}
From Proposition \ref{prop:3}, we have that
$$v(-1/w)=\bar A(v(w/4))=-B(-v(w/4)),$$
where
$$\displaystyle B(x)=\frac{x+\sigma}{\sigma x-1}=-\frac{-(-x)+\sigma}{\sigma(-x)+1}=-\bar A(-x).$$
Hence, according to (\ref{eqn:7.12}), we obtain
$$v(w/8)^{\tau_2}=\frac{v(w/4)-1}{v(w/4)+1}=B(v(-1/w))\ \ \iff\ \ v(-1/w)=B(v(w/8)^{\tau_2}),$$
showing that both the statements in the theorem are equivalent. We now show that Proposition \ref{prop:12} implies that $v(w/8)$ and $v(-1/w)$ are conjugate algebraic integers.\medskip

In similar fashion to (\ref{eqn:7.6}), we set
$$\eta=v(w/8),\ \ \tilde \lambda=v(-1/w) = -\lambda,\ \ \nu=v(w/4).$$
Then, according to (\ref{eqn:7.7}), we get
$$\frac{2 \tilde \lambda}{1-\tilde \lambda^2}=-\frac{2-(\frac{1}{\nu}-\nu)}{2+(\frac{1}{\nu}-\nu)}=-\frac{2-\frac{2}{\mathfrak{p}(2w)}}{2+\frac{2}{\mathfrak{p}(2w)}}=-\frac{1+\pi^{\tau_2}}{1-\pi^{\tau_2}}=\frac{\pi^{\tau_2}+1}{\pi^{\tau_2}-1}.$$
Since $\frac{\pi^{\tau_2}+1}{\pi^{\tau_2}-1}$ is a root of $b_d(x)$, we have that
$$f_d(\tilde \lambda)=2^{-h}(\tilde \lambda^2-1)^{2h}b_d\left(\frac{2\tilde \lambda}{1-\tilde \lambda^2}\right)=0,$$
showing that $\tilde \lambda = v(-1/w)$ is a conjugate of $\eta = v(w/8)$.\medskip

Now,
\begin{align*}
B(\tilde \lambda)-\eta^2&=\frac{\tilde \lambda+\sigma}{\sigma \tilde \lambda-1}-\frac{\sigma^2-\tilde \lambda^2}{1-\sigma^2 \tilde \lambda^2} = \frac{\lambda-\sigma}{\sigma \lambda+1}-\frac{\sigma^2- \lambda^2}{1-\sigma^2 \lambda^2}\\
&=\frac{(\lambda-\sigma)(\sigma \lambda-1)+(\sigma^2- \lambda^2)}{\sigma^2 \lambda^2-1}\\
&=\frac{(\sigma-1) \lambda^2-(\sigma^2+1) \lambda+(\sigma^2+\sigma)}{(\sigma \lambda+1)(\sigma \lambda-1)}\\
&=\frac{-\sqrt{2}\,\sigma(\lambda^2+ 2\lambda-1)}{\sigma^2(\lambda-\bar\sigma)(\lambda+\bar\sigma)}\\
&=\frac{-\sqrt{2}\,\sigma(\lambda-\sigma)(\lambda-\bar\sigma)}{\sigma^2(\lambda-\bar\sigma)(\lambda+\bar\sigma)}\\
&=\frac{\sqrt{2}\,\bar\sigma(\tilde \lambda+\sigma)}{(\tilde \lambda-\bar\sigma)}.
\end{align*}
In the above calculation, mapping $\sqrt{2}$ to $-\sqrt{2}$, while fixing $\tilde \lambda$, gives us
$$\bar B(\tilde \lambda)-\frac{1}{\eta^2}=-\frac{\sqrt{2}\,\sigma(\tilde \lambda+\bar\sigma)}{(\tilde \lambda-\sigma)}.$$
Multiplying the above two expressions gives us
\begin{align*}
\big(B(\tilde \lambda)-\eta^2\big)\left(\bar B(\tilde \lambda)-\frac{1}{\eta^2}\right)&=2\,\frac{(\tilde \lambda+\sigma)(\tilde \lambda+\bar\sigma)}{(\tilde \lambda-\sigma)(\tilde \lambda-\bar\sigma)} =2 \frac{\tilde \lambda^2-2\tilde \lambda-1}{\tilde \lambda^2+2\tilde \lambda-1}\\
&=2\frac{1+\big(\frac{2\tilde \lambda}{1-\tilde \lambda^2}\big)}{1-\big(\frac{2\tilde \lambda}{1-\tilde \lambda^2}\big)}=2\frac{1+\big(\frac{\pi^{\tau_2}+1}{\pi^{\tau_2}-1}\big)}{1-\big(\frac{\pi^{\tau_2}+1}{\pi^{\tau_2}-1}\big)} =-2 \pi^{\tau_2}.
\end{align*}
Now a similar argument to the end of the proof of Theorem \ref{thm:3} applies here and shows that the automorphism $\psi$ on $F_1$ taking $\eta$ to $\tilde \lambda$ is $\eta^\psi = B(\tilde \lambda)$.  As before, $\psi$ coincides with $\tau_2$, giving that $\tilde \lambda = v(-1/w) = B(\eta^{\tau_2}) = B(v(w/8)^{\tau_2})$.  Also see the argument below.  \end{proof}

\noindent {\bf Corollary 1.} {If $c$ is even, the point $(x,y) = (-\eta,-\eta^{\tau_2})$ lies on the curve $f(x, y) = 0$:}
$$f(-v(w/8),-v(w/8)^{\tau_2})=0,\ \ \tau_2=\left(\frac{\Sigma_{\wp_2'^3}\Omega_f/K}{\wp_2}\right).$$

\begin{proof} We have
\begin{align*}
0 &= f(v(w/8),v(w/4)) = f\left(v(w/8),-\frac{v(w/4)-1}{v(w/4)+1}\right)\\
& = f(v(w/8),-v(w/8)^{\tau_2}) = f(-v(w/8),-v(w/8)^{\tau_2}).
\end{align*}
\end{proof}

\noindent {\bf Corollary 2.} {\it The field $\mathbb{Q}(v(-1/w)) = \Sigma_{\wp_2^3}\Omega_f$ is the inertia field for the prime ideal $\wp_2'$ in the extension $L_{\mathcal{O},8}/K$.} \medskip

We also give an alternate argument to show $\psi = \tau_2$ in the proofs of Theorems \ref{thm:3} and \ref{thm:4}.  We first note that the modular function $j(\tau)$ can be expressed in terms of $z = v(\tau)-\frac{1}{v(\tau)}$, namely
$$j(\tau) = J(z) = \frac{(z^8 + 240z^6 + 2144z^4 + 3840z^2 + 256)^3}{z^2(z^2 + 4)^2 (z - 2)^8 (z + 2)^8},$$
using Proposition \ref{prop:9}.  Now set $z = \eta - \frac{1}{\eta} = \pm \frac{2}{\pi}$, so that $(z, \wp_2)=1$.  This allows us to reduce the above formula modulo $\wp_2$, giving that
$$j(w/8) \equiv \frac{z^{24}}{z^{22}} \equiv z^2 \ (\textrm{mod} \ \wp_2).$$
This shows that $j(w/8)^\tau$ is conjugate to $z^\tau$ modulo each prime divisor $\mathfrak{p}$ of $\wp_2$ in $\Omega_f$, for each automorphism $\tau \in \textrm{Gal}(\Omega_f/K)$; and this implies that the class equation $H_{-d}(X)$ and the minimal polynomial $\mu_d(X)$ of $z$ over $K$ are congruent:
$$H_{-d}(X) \equiv \mu_d(X) \ (\textrm{mod} \ \wp_2).$$
A theorem of Deuring says that the discriminant of $H_{-d}(X)$ is odd (since $\displaystyle \left(\frac{-d}{2}\right) = +1$), so the discriminant of $\mu_d(X)$ is not divisible by $\wp_2$.  This implies that the discriminant of the minimal polynomial
$\tilde \mu_d(X) = X^{h(-d)}\mu_d\left(X-\frac{1}{X}\right)$ of $\eta$ over $K$ is relatively prime to $\wp_2$, as well.  This is because
$$\mu_d(X) = \prod_{i=1}^{h(-d)}{(X-(\eta_i-\frac{1}{\eta_i}))}$$
is a product over the conjugates $z_i = \eta_i-\frac{1}{\eta_i}$ of $z$, so that
\begin{align*}
X^{h(-d)} \mu_d\left(X-\frac{1}{X}\right) &= \prod_{i=1}^{h(-d)}{(X^2-(\eta_i-\frac{1}{\eta_i})X-1)},\\
& = \prod_{i=1}^{h(-d)}{(X^2-z_i X-1)}, \ \ z_i = \eta_i-\frac{1}{\eta_i}.
\end{align*}
Hence,
\begin{align*}
\textrm{disc}(\tilde \mu_d(X)) & = \prod_{i=1}^{h(-d)}{(z_i^2+4)} \ \prod_{i < j}{\textrm{Res}(X^2-z_i X-1,X^2-z_j X-1)^2}\\
& = \prod_{i=1}^{h(-d)}{(z_i^2+4)} \ \prod_{i < j}{(z_i-z_j)^4}\\
& = \prod_{i=1}^{h(-d)}{(z_i^2+4)} \ (\textrm{disc}(\mu_d(X)))^2.
\end{align*}
Now the $z_i$ are conjugate over $K$, so each $z_i$ is relatively prime to $\wp_2$, which implies that $(z_i^2+4, \wp_2) = 1$ for each $i$.  This proves the claim that $(\textrm{disc}(\tilde \mu_d(X)), \wp_2) = 1$.  This proves

\begin{thm} Let $\textsf{R}_{\wp_2}$ denote the ring of elements of $K$ which are integral for $\wp_2$.  Then the powers of $\eta = v(w/8)$ form a basis over $\textsf{R}_{\wp_2}$ for the ring $\overline{\textsf{R}}$ of elements of $F_1 = \mathbb{Q}(\eta)$ which are integral for $\wp_2$.
\label{thm:5}
\end{thm}

Given this theorem, the congruence
$$\eta^\psi \equiv \eta^2 \ (\textrm{mod} \ \wp_2)$$
implies that
$$\alpha^\psi \equiv \alpha^2 \ (\textrm{mod} \ \wp_2),$$
for all $\alpha \in F_1$ which are integral for $\wp_2$.  Since $F_1/K$ is abelian and $\wp_2$ is unramified in this extension, this implies by definition of the Artin symbol that $\psi = \tau_2$.

\section{Values of $v(\tau)$ as periodic points.}

We now define the following algebraic functions.  The roots of $f(x,y) = y^2+(x^2-1)y+x^2$ (see Proposition \ref{prop:1}) as a function of $y$ are 
\begin{equation}
\hat F(x) = -\frac{x^2-1}{2} \pm \frac{1}{2} \sqrt{x^4 - 6x^2 + 1}.
\label{eqn:8.1}
\end{equation}
Also, the roots of $g(x,y) = y^2-(x^2-4x+1)y+x^2$ (see Proposition \ref{prop:2}) are given by
\begin{align}
\notag \hat T(x) &= \frac{1}{2}(x^2-4x+1) \pm \frac{1}{2} \sqrt{(x^2-2x+1)(x^2-6x+1)}\\
\label{eqn:8.2} & = \frac{1}{2}(x^2-4x+1) \pm \frac{x-1}{2} \sqrt{x^2-6x+1}.
\end{align}

We now prove the following.

\begin{thm} If $w \in R_K$ is the algebraic integer defined by
$$w = \frac{a+\sqrt{-d}}{2}, \ \ \textrm{with} \ a^2+d \equiv 0 \ (\textrm{mod} \ 2^5)$$
and the integer $c$ satisfies
$$c \equiv 1-\frac{a^2+d}{32} \ \ (\textrm{mod} \ 2),$$
then the generator $(-1)^{1+c}v(w/8)$ of the field $\Sigma_{\wp_2'^3} \Omega_f$ over $\mathbb{Q}$ is a periodic point of the algebraic function $\hat F(x)$ defined by (\ref{eqn:8.1}) and $v^2(w/8)$ is a periodic point of the function $\hat T(x)$ defined by (\ref{eqn:8.2}).
\label{thm:6}
\end{thm}
\begin{proof}
Setting $\eta = (-1)^{1+c}v(w/8)$ and $F_1 = \mathbb{Q}(\eta) = \mathbb{Q}(\eta^2)$, we have from the corollaries to Theorems \ref{thm:3} and \ref{thm:4} that $f(\eta,\eta^{\tau_2}) = 0$, where $\displaystyle \tau_2 = \left(\frac{F_1/K}{\wp_2}\right)$ is an automorphism in $\textrm{Gal}(F_1/K)$.  If the order of $\tau_2$ is $n$, then applying powers of $\tau_2$ gives that
\begin{equation}
f(\eta, \eta^{\tau_2}) = f(\eta^{\tau_2}, \eta^{\tau_2^2}) = \cdots =  f(\eta^{\tau_2^{n-1}}, \eta) = 0,
\label{eqn:8.3}
\end{equation}
which implies that $\eta$ is a periodic point of $\hat F(x)$ of period $n$. \medskip

It is straightforward to check that
\begin{equation}
\hat F(x)^2 = \frac{1}{2}(x^4-4x^2+1) \pm \frac{1}{2}(x^2-1) \sqrt{x^4-6x^2+1} = \hat T(x^2)
\label{eqn:8.4}
\end{equation}
and that the minimal polynomial of $\hat F(x)^2$ over $\mathbb{Q}(x)$ is $g(x^2,y)$.
In particular, $f(x,y) = 0$ implies that $g(x^2,y^2) = 0$, since
$$g(x^2,y^2) = (-x^2 y + x^2 + y^2 + y) (x^2 y + x^2 + y^2 - y) = f(x,-y) f(x,y).$$
Hence, (\ref{eqn:8.3}) implies that
\begin{equation}
g(\eta^2, \eta^{2\tau_2}) = g(\eta^{2\tau_2}, \eta^{2\tau_2^2}) = \cdots =  g(\eta^{2\tau_2^{n-1}}, \eta^2) = 0,
\label{eqn:8.5}
\end{equation}
which shows that $\eta^2 = v(w/8)^2$ is a periodic point of $\hat T(x)$.
\end{proof}

As in the papers \cite{m1}-\cite{m3}, the minimal polynomials of periodic points of $\hat F(x)$ can be computed using iterated resultants involving its minimal polynomial $f(x,y)$.  We set
$$R^{(1)}(x,x_1) = f(x, x_1) = x^2 x_1 + x^2 + x_1^2 - x_1$$
and define, inductively,
$$R^{(n)}(x,x_n) = \textrm{Res}_{x_{n-1}}(R^{(n-1)}(x,x_{n-1}),f(x_{n-1},x_n)) \ \ n \ge 2.$$
Then the roots of the polynomial
$$R_n(x) = R^{(n)}(x,x), \ \ n \ge 1,$$
are the periodic points of $\hat F(x)$ whose minimal periods divide $n$.  See \cite[p. 727]{m1}.  For example, we compute that
\begin{align*}
R_1(x) & = x(x^2 + 2x - 1),\\
R_2(x) & = x(x^2 + 2x - 1)(x^4 - x^3 + x + 1),\\
R_3(x) & = x(x^2 + 2x - 1)(x^{12} - 5x^{11} + 2x^{10} + 10x^9 + 5x^8 + 23x^7\\
& \ \ \  - 8x^6 - 23x^5 + 5x^4 - 10x^3 + 2x^2 + 5x + 1),\\
R_4(x) & = x(x^2 + 2x - 1)(x^4 - x^3 + x + 1)(x^8 - x^7 + x^6 - 5x^5 + 5x^3 + x^2 + x + 1)\\
& \ \ \ \times (x^{16} + 5x^{15} - 18x^{14} - 75x^{13} + 137x^{12} + 105x^{11} + 38x^{10} + 185x^9\\
& \ \ \  - 300x^8 - 185x^7 + 38x^6 - 105x^5 + 137x^4 + 75x^3 - 18x^2 - 5x + 1).
\end{align*}

We now set $x = z+3$ in the function $\hat T(x)$, so that the square-root in $\hat T(x)$ has the $2$-adic expansion
$$\sqrt{x^2-6x+1} = \sqrt{z^2-8} = z \sqrt{1-\frac{8}{z^2}} = z \sum_{k=0}^\infty{(-1)^k\left({1/2 \atop k}\right) \frac{8^k}{z^{2k}}}.$$
We will show that this series is $2$-adically convergent for (roughly) half of the primitive periodic points of the algebraic function $\hat T(x)$ of a given period $n$ in the field $\textsf{K}_2(\sqrt{2})$, where $\textsf{K}_2$ is the maximal unramified, algebraic extension of the $2$-adic field $\mathbb{Q}_2$. \medskip

If we set
$$T(x) =  \frac{1}{2}(x^2-4x+1) + \frac{x-1}{2} \sqrt{x^2-6x+1},$$
then using the above series in $T(x)$ and splitting off the $k=0$ term, we find
$$T(x) = x^2 - 4x + 2 + (x-1)(x-3) \sum_{k=1}^\infty{(-1)^k 2^{2k-1}\left({1/2 \atop k}\right) \frac{2^k}{(x-3)^{2k}}},$$
for $x-3 \in \mathcal{O}^\times$, where $\mathcal{O}$ is the ring of integers in $\textsf{K}_2(\sqrt{2})$.  Since
$$(-1)^{k-1} 2^{2k-1}\left({1/2 \atop k}\right) = C_{k-1} \in \mathbb{Z}$$
is the Catalan sequence, it follows that
$$T(x) \equiv x^2 \ (\textrm{mod} \ 2), \ \ x-3 \in \mathcal{O}^\times.$$
Hence, $T(x)$ is a lift of the Frobenius automorphism for points $x$ in the set
$$\overline{\textsf{D}} = \{x \in \textsf{K}_2(\sqrt{2}): |x-3|_2 = 1\}.$$
Furthermore,
$$T(x) - 3 = (x-3)^2+2(x -3) - 4 - (x-1)(x-3) \sum_{k=1}^\infty{C_{k-1} \frac{2^k}{(x-3)^{2k}}}.$$
It follows that
\begin{equation}
|T(x)-3|_2 = |x-3|_2^2 = 1,
\label{eqn:8.6}
\end{equation}
and $T$ maps $\overline{\textsf{D}}$ to itself. \medskip

Now we prove

\begin{prop} We have the congruences
\begin{align*}
R^{(n)}(x,x_n) & \equiv (x^{2^n} + x_n)(x_n + 1)^{2^n-1} \ (\textrm{mod} \ 2);\\
R_n(x) & \equiv (x^{2^n} + x)(x + 1)^{2^n-1} \ (\textrm{mod} \ 2).
\end{align*}
\label{prop:13}
\end{prop}

\begin{proof}
We have $f(x,y)=x^2y+x^2+y^2-y$.\ \ So, for $n=1$, we get
\begin{align*}
R^{(1)}(x,x_1)=f(x,x_1)&=x^2x_1+x^2+x_1^2-x_1\\
 & \equiv x^2x_1+x^2+x_1^2+x_1\ (\textrm{mod}\ 2)\\
 & \equiv (x^2+x_1)\,(x_1+1)\ (\textrm{mod}\ 2).
\end{align*}
Hence,
$$R_1(x) \equiv (x^2+x)\,(x+1)\ (\textrm{mod}\ 2).$$
Now for the induction step, assume the result is true for $n-1$. Then,
\begin{align*}
R^{(n)}(x,x_n) & = \textrm{Res}_{x_{n-1}}(R^{(n-1)}(x,x_{n-1}),f(x_{n-1},x_n))\\
&\equiv\textrm{Res}_{x_{n-1}}((x^{2^{n-1}} + x_{n-1})(x_{n-1} + 1)^{2^{n-1}-1},(x_{n-1}^2+x_n)(x_n+1))\ \ (\textrm{mod}\ 2).
\end{align*}
By definition, the resultant of two polynomials $\displaystyle f=\sum_{i=0}^n{a_ix^i}$ and $\displaystyle g=\sum_{i=0}^m{b_ix^i}$, having roots $\alpha_1,\alpha_2,\ldots,\alpha_n$ and $\beta_1,\beta_2,\ldots,\beta_m$, respectively, is given by
$$\textrm{Res}(f,g)=a_n^m\prod_{i=1}^n{g(\alpha_i)},$$
and
$$\textrm{Res}(g,f)=(-1)^{mn}\textrm{Res}(f,g).$$
The roots of $(x_{n-1}^2+x_n)\,(x_n+1)$, as a polynomial in $x_{n-1}$, are $\pm\sqrt{-x_n}$. Hence,
\begin{align*}
\textrm{Res}_{x_{n-1}}&((x^{2^{n-1}} + x_{n-1})(x_{n-1} + 1)^{2^{n-1}-1},(x_{n-1}^2+x_n)(x_n+1))\\
&=(-1)^{2^{n-1}\cdot2}\big(x_n+1\big)^{2^{n-1}} \big(x^{2^{n-1}}+\sqrt{-x_n}\big)\big(\sqrt{-x_n}+1\big)^{2^{n-1}-1}\\
& \ \ \ \times\big(x^{2^{n-1}}-\sqrt{-x_n}\big) \big(-\sqrt{-x_n}+1\big)^{2^{n-1}-1}\\
&=(-1)^{2^{n}}\big(x_n+1\big)^{2^{n-1}} \big(x^{2^n}+x_n\big) \big(x_n+1\big)^{2^{n-1}-1}\\
&=\big(x^{2^n}+x_n\big) \big(x_n+1\big)^{2^n-1}.
\end{align*}
Hence, we obtain
\begin{align*}
R^{(n)}(x,x_n) & \equiv (x^{2^n} + x_n)(x_n + 1)^{2^n-1} \ (\textrm{mod} \ 2),\\
R_n(x) & \equiv (x^{2^n} + x)(x + 1)^{2^n-1} \ (\textrm{mod} \ 2),
\end{align*}
completing the induction.
\end{proof}

\noindent {\bf Corollary.} {\it The degree of $R_n(x)$ is $\textrm{deg}(R_n(x)) = 2^{n+1}-1$.} \medskip

\begin{proof} This follows from the proposition, if the leading coefficient of $R_n(x)$ is not divisible by $2$.  We refer the reader to the lemma in \cite[pp. 727-728]{m1} for a similar proof.
\end{proof}

The roots of the factor $x^{2^n} + x = x(x+1)\frac{x^{2^n-1}+1}{x+1} = x(x+1)h_n(x)$ other than $x=0,1$ have degree greater than $1$, and therefore satisfy $x-3 \not \equiv 0$ (mod $2$).  It follows from Hensel's Lemma that $2^n-1$ of the roots of $R_n(x)$ over $\mathbb{Q}_2$ have the property that $x-3 \in \mathcal{O}^\times$, and for these roots the series for $T(x)$ converges in $\textsf{K}_2$. \medskip

Now the argument at the end of the proof of Theorem \ref{thm:3} shows that $\eta \not \equiv 1$ (mod $\wp_2$), so that the image of $\eta$ in the completion $F_{1,\mathfrak{q}} \subset \textsf{K}_2$ of $F_1= \Sigma_{\wp_2'^3} \Omega_f$ with respect to a prime divisor $\mathfrak{q}$ of $\wp_2$ in $F_1$ satisfies $\eta^2-3 \in \mathcal{O}^\times$.  Hence, the series for $T(\eta^2)$ converges.  We claim now that $\eta^{2\tau_2} = T(\eta^2)$.  But $g(\eta^2,\eta^{2\tau_2}) = 0$ implies that $\eta^{2\tau_2}$ is one of the values of $\hat T(\eta^2)$.  The value different from $T(\eta^2)$ in $\textsf{K}_2$ is
\begin{align*}
T_1(\eta^2) & = \eta^4-4\eta^2+1-T(\eta^2)\\
& \equiv \eta^4-4\eta^2+1-\eta^4 \ (\textrm{mod} \ \mathfrak{q})\\
& \equiv 1 \ (\textrm{mod} \ \mathfrak{q}).
\end{align*}
But we also know $\eta^{2\tau_2}-3 = (\eta^2-3)^{\tau_2} \in \mathcal{O}^\times$, so that $\eta^{2\tau_2} \neq T_1(\eta^2)$.  This yields the following.

\begin{thm} If $w$ satisfies (\ref{eqn:6.1}), then the value $\eta = v(w/8)$ and the automorphism $\displaystyle \tau_2 = \left(\frac{F_1/K}{\wp_2}\right)$ satisfy
$$\eta^{2\tau_2} = T(\eta^2),$$
in the completion $F_{1,\mathfrak{q}} \subset \textsf{K}_2$ of $F_1= \Sigma_{\wp_2'^3} \Omega_f$ with respect to a prime divisor $\mathfrak{q}$ of $\wp_2$ in $F_1$, where
$$T(x) = x^2 - 4x + 2 - (x-1)(x-3) \sum_{k=1}^\infty{C_{k-1} \frac{2^k}{(x-3)^{2k}}}$$
converges for $x$ in $\overline{\textsf{D}} = \{x \in \textsf{K}_2(\sqrt{2}): |x-3|_2 = 1\}$.
\label{thm:7}
\end{thm}

Since $\tau_2$ fixes the prime divisors of $\wp_2$, it extends naturally to an automorphism of $F_{1,\mathfrak{q}}$, and can be applied to the individual terms of the series representing $T(x)$.  Thus, we see inductively that
$$\eta^{2\tau_2^i} = T(\eta^{2\tau_2^{i-1}}) = T(T^{i-1}(\eta^2)) = T^i(\eta^2)$$
is the $i$-th iterate of $T(x)$ applied to $\eta^2$.  From this and the fact that $\mathbb{Q}(\eta^2) = F_1$ we see that the order of $\tau_2$ in $\textrm{Gal}(F_1/K)$ is the minimal period of the periodic point $\eta^2$, and that $\eta^2$ is a periodic point in the ordinary sense of the $2$-adic function $T(x)$.  This also shows that the minimal period of $\eta$ with respect to $\hat F(x)$ is $n = \textrm{ord}(\tau_2)$, since if $\eta$ had smaller minimal period $m$, then by the proof of Theorem \ref{thm:6}, $\eta^2$ would have period $m < n$ with respect to the function $T(x)$.  This completes the proof of the assertions of Theorem B of the Introduction regarding minimal periods.

\section{The periodic points of $\hat F(x)$ and a class number formula.}

In this section we show that the only periodic points of $\hat F(x)$ are the values given in Theorem \ref{thm:6}.  In fact, we will
prove the following.

\begin{thm} The only periodic points of the function $\hat F(x)$ in $\overline{\mathbb{Q}}$ are the fixed points $0, \sigma, \bar \sigma$ and the conjugates over $\mathbb{Q}$ of the values $v(w/8)$ in Theorem \ref{thm:6} (for odd $c$).
\label{thm:8}
\end{thm}

\begin{proof} Let $\tilde g(x,y) = x^2 y^2 +2y+x^2$.  Note that $\tilde g(x,y) = g(y,x)$ for the polynomial $g(x,y)$ in \cite[Thm. 2, p. 327]{m}.  By the results of that paper the numbers $\pi, \xi$ and their conjugates over $\mathbb{Q}$ (as $-d$ ranges over all discriminants $\equiv 1$ modulo $8$) are, together with $0$ and $-1$, the only periodic points of the algebraic function $\mathfrak{f}(z)$ defined by $\tilde g(z,\mathfrak{f}(z)) = 0$.  The assertion of the theorem will follow from the identity
\begin{equation}
(x^2-1)^2 (y^2-1)^2 \tilde g\left(\frac{2x}{x^2-1},\frac{2y}{y^2-1}\right) = 4f(x,y)(x^2y^2-x^2y+y+1).
\label{eqn:9.1}
\end{equation}
Here, as in Proposition \ref{prop:1}, $f(x,y) = x^2 y+x^2+y^2 - y$.  Let $\eta$ be a periodic point of $\hat F(x)$ in $\overline{\mathbb{Q}}$ which is distinct from its fixed points $0, \sigma, \bar \sigma$.  Then there are $\eta_1=\eta, \eta_2, \dots, \eta_n$ in $\overline{\mathbb{Q}}$ for which
\begin{equation}
f(\eta_1,\eta_2) = f(\eta_2,\eta_3) = \cdots = f(\eta_n,\eta_1) = 0.
\label{eqn:9.2}
\end{equation}
Setting $\lambda_i = \frac{2\eta_i}{\eta_i^2-1}$, equations (\ref{eqn:9.1}) and (\ref{eqn:9.2}) give that
\begin{equation}
\tilde g(\lambda_1,\lambda_2) = \tilde g(\lambda_2,\lambda_3) = \cdots = \tilde g(\lambda_n,\lambda_1) = 0.
\label{eqn:9.3}
\end{equation}
Note that $\eta_i \neq \pm 1$ since $\pm 1$ are preperiodic (and not periodic) for $f(x,y)$, since
$$f(\pm1,y) = y^2+1, \ \ f(\pm i,y) = y^2-2y-1, \ \ f(1 \pm \sqrt{2},y) = (y+1 \pm \sqrt{2})^2.$$
Equation (\ref{eqn:9.3}) implies that $\lambda_1$ is a periodic point of the function $\mathfrak{f}(z)$ defined above.  Also, $\lambda_i \neq 0, -1$ since $\eta_i \notin \{0, \sigma, \bar \sigma\}$.  By the results of \cite[Thm. 2]{m}, this shows that $\lambda_1$ must be a conjugate of the number $\pi$ for some discriminant $-d$ and is therefore a root of the polynomial $b_d(x)$.  (See Proposition \ref{prop:10}.)  Since $\lambda_1 = 2\eta/(\eta^2-1)$, this shows that $\eta$ is a root of the minimal polynomial $f_d(x)$ of $v(w/8)$, for $c$ odd, by (\ref{eqn:6.7}).  This completes the proof.
\end{proof}

The reason it suffices to take $c$ odd in this theorem is that if $c$ is even, meaning that $2^5 \ || \ a^2 +d$, then $2^6 \mid (a+16)^2+d$, so that $w + 8= \frac{a+16+\sqrt{-d}}{2} = w'$ satisfies (\ref{eqn:6.1}) with $c$ odd.  Then the infinite product formula for $v(\tau)$ shows that $v(w/8) = v(w'/8-1) = -v(w'/8)$, so that $-v(w/8) = v(w'/8)$ in Corollary 1 to Theorem \ref{thm:4}. \medskip

This theorem has the following consequence.  Let $F_1 = \Sigma_{\wp_2'^3}\Omega_f$ be the field generated by $v(w/8)$ in Theorem \ref{thm:1}.  Then $[F_1:\mathbb{Q}]=4h(-d)$ and the prime ideal $\wp_2$ is not ramified in $F_1/\mathbb{Q}$.  The field $F_1$ is the inertia field for $\wp_2$ in the field $\Sigma_8\Omega_f$, an extended ring class field over $K_d = \mathbb{Q}(\sqrt{-d})$.  As in Section 7, let $\displaystyle \tau_2 = \left(\frac{F_1/K_d}{\wp_2}\right)$ be the Artin symbol for $\wp_2$ in the extension $F_1/K_d$.  Now define the set of discriminants
\begin{equation}
\mathfrak{D}_{n,2} = \{-d < 0 \ |  -d \equiv 1 \ (\textrm{mod} \ 8) \ \textrm{and} \ \textrm{ord}(\tau_2) = n \ \textrm{in} \ \textrm{Gal}(F_1/K_d)\}.
\label{eqn:9.4}
\end{equation}

\begin{thm} If $n \ge 2$, we have the following relation between class numbers of discriminants in the set $\mathfrak{D}_{n,2}$:
\begin{equation}
\sum_{-d \in \mathfrak{D}_{n,2}}{h(-d)} = \frac{1}{2} \sum_{k \mid n}{\mu(n/k) 2^k}.
\label{eqn:9.5}
\end{equation}
\label{thm:9}
\end{thm}

\begin{proof}
This proof mirrors the arguments in \cite[pp.792-793, 806]{m3}.  First, define
\begin{equation}
\textsf{P}_n(x) = \prod_{k \mid n}{R_k(x)^{\mu(n/k)}}.
\label{eqn:9.6}
\end{equation}
We show that $\textsf{P}_n(x) \in \mathbb{Z}[x]$.  From Proposition \ref{prop:13} it is clear that $R_n(x)$, for $n > 1$, is divisible (mod $2$) by the $N$ irreducible (monic) polynomials $\bar h_i(x)$ of degree $n$ over $\mathbb{F}_2$,
where
$$N= \frac{1}{n}\sum_{k \mid n}{\mu(n/k)2^k},$$
and that these polynomials are simple factors of $R_n(x)$ (mod $2$).  It follows from Hensel's Lemma that $R_n(x)$ is divisible by distinct irreducible polynomials $h_i(x)$ of degree $n$ over $\mathbb{Z}_2$, the ring of integers in $\mathbb{Q}_2$, for $1 \le i \le N$, with $h_i(x) \equiv \bar h_i(x)$ (mod $2$). In addition, all the roots of $h_i(x)$ are periodic of minimal period $n$ and lie in the unramified extension $\textsf{K}_2$.  Furthermore, $n$ is the smallest index for which $h_i(x) \mid R_n(x)$ over $\mathbb{Q}_2$. \medskip

Now consider the identity
\begin{equation}
(\sigma x+1)^2(\sigma y+1)^2 f(\bar A(x),\bar A(y)) = 2^3 \sigma^2 f(y,x),
\label{eqn:9.7}
\end{equation}
where $\displaystyle \bar A(x) = \frac{-x+\sigma}{\sigma x+1}$, as in (\ref{eqn:3.3}).  If the periodic point $a$ of $\hat F(x)$, with minimal period $n > 1$, is a root of one of the polynomials $h_i(x)$, then $a$ is a unit in $\textsf{K}_2$, and for some $a_1, \dots, a_{n-1}$ we have
\begin{equation}
f(a,a_1)=f(a_1,a_2) = \cdots = f(a_{n-1},a)=0.
\label{eqn:9.8}
\end{equation}
Furthermore, $a \not \equiv 1 \ (\textrm{mod} \ \sqrt{2})$, since otherwise its reduction $a \equiv \bar a \equiv 1$ (mod $2$) would have degree $1$ over $\mathbb{F}_2$ (using that $\textsf{K}_2$ is unramified over $\mathbb{Q}_2$).  Hence, $a+1+\sqrt{2}$ is a unit in $\textsf{K}_2(\sqrt{2})$, which gives that $\sigma a+1$ is a unit, as well.  All of the $a_i$ satisfy $a_i \not \equiv 1$ (mod $\sqrt{2}$), since the congruence $f(1,y) \equiv (y+1)^2$ (mod $2$) has only $y \equiv 1$ as a solution.  Hence, if some $a_i \equiv 1$ (mod $\sqrt{2}$), then $a_j \equiv 1$ for $j >i$, which would imply that $a \equiv 1$ (mod $\sqrt{2}$), as well.  The elements $b_i=\bar A(a_i)$ are distinct and lie in $\textsf{K}_2(\sqrt{2})$ and satisfy
$$b_i - 1 \equiv \frac{-a_i+\sigma-\sigma a_i-1}{\sigma a_i+1} \equiv \frac{-2}{\sigma a_i + 1} \equiv 0 \ (\textrm{mod} \ \sqrt{2}).$$
The identity (\ref{eqn:9.7}) yields that
\begin{equation}
f(b,b_{n-1})=f(b_{n-1},b_{n-2})= \dots = f(b_1,b)=0
\label{eqn:9.9}
\end{equation}
in $\textsf{K}_2(\sqrt{2})$.  Hence, $b_i \equiv 1$ (mod $\sqrt{2}$), and the orbit $\{b,b_{n-1},\dots,b_1\}$ is distinct from all the orbits in (\ref{eqn:9.8}). \medskip

Now the map $\bar A(x)$ has order $2$, so it is clear that $b=\bar A(a)$ has minimal period $n$ in (\ref{eqn:9.9}), since otherwise $a=\bar A(b)$ would have period smaller than $n$.  It follows that there are at least $2N$ periodic orbits of minimal period $n>1$.  Noting that
$$R_1(x)= f(x,x) = \ x(x^2 + 2x - 1),$$
these distinct orbits and factors account for at least
\begin{equation*}
3 + \sum_{d \mid n,d>1}({2 \sum_{k \mid d}{\mu(d/k)2^k})} = -1+ 2\sum_{d \mid n}({\sum_{k \mid d}{\mu(d/k)2^k})} = 2 \cdot 2^n-1
\end{equation*}
roots, and therefore all the roots, of $R_n(x)$.  This shows that the roots of $R_n(x)$ are distinct and the expressions $\textsf{P}_n(x)$ are polynomials.  Furthermore, over $\textsf{K}_2(\sqrt{2})$ we have the factorization
\begin{equation}
\textsf{P}_n(x) = \pm \prod_{1 \le i \le N}{h_i(x) \tilde h_i(x)}, \ \ n>1,
\label{eqn:9.10}
\end{equation}
where $\tilde h_i(x) = c_i(\sigma x+1)^{n}h_i(\bar A(x))$, and the constant $c_i$ is chosen to make $\tilde h_i(x)$ monic. \medskip

By the results of Section 8, for each discriminant $-d \in \mathfrak{D}_{n,2}$ we have that $f_d(x) \mid \textsf{P}_n(x)$.  Furthermore, every root of $\textsf{P}_n(x)$ is a root of some $f_d(x)$, by Theorem \ref{thm:8}, where $ord(\tau_2) = n$ in order for the roots of $f_d(x)$ to have minimal period $n$.  It follows that
\begin{equation*}
\textsf{P}_n(x) = \tilde c_n \prod_{-d \in \mathfrak{D}_{n,2}}{f_d(x)},
\end{equation*}
for some constant $\tilde c_n$, and taking degrees on both sides and using (\ref{eqn:9.10}) gives the formula
$$2\sum_{k \mid n}{\mu(n/k) 2^k} = \sum_{-d \in \mathfrak{D}_{n,2}}{4h(-d)}.$$
The formula of the theorem follows.
\end{proof}

The result of Theorem \ref{thm:9} is the analogue of \cite[Thm.1.3]{m3} for the prime $2$ in place of $5$.  The factor $1/2$ in front is to be interpreted as $2/\phi(8)$, replacing the factor $2/\phi(5)$ in the result of \cite{m3}.  Also, see Conjecture 1 in the Introduction of that paper. \medskip

Theorem 8 will now be used to prove the corresponding fact for the algebraic function $\hat T(x)$ in Theorem 6.

\begin{thm}
The periodic points of the function $\hat T(x)$ of (\ref{eqn:8.2}) in $\overline{\mathbb{Q}}$ (or $\mathbb{C}$) are exactly the squares of the periodic points of the function $\hat F(x)$, i.e., the fixed points $0,\sigma^2,\bar\sigma^2$ and the conjugates over $\mathbb{Q}$ of the values $v ^2(w/8)$, where $w$ is given by (\ref{eqn:6.1}).
\label{thm:10}
\end{thm}

\begin{proof} As in the proof of Theorem \ref{thm:6}, the polynomials $g(x,y)=y^2-(x^2-4x+1)y+x^2$ and $f(x,y)=y^2+(x^2-1)y+x^2$ defining $\hat T$ and $\hat F$, respectively, satisfy the identity
$$g(x^2,y^2)=f(x,-y) f(x,y).$$
Let $\eta^2$ be a periodic point of $g(x,y)$ of period $n$. Then there exist $\eta_1^2,\eta_2^2,\ldots,\eta_{n-1}^2 \in \overline{\mathbb{Q}}$ such that
$$g(\eta^2,\eta_1^2)=g(\eta_1^2,\eta_2^2)=\cdots=g(\eta_{n-1}^2,\eta^2)=0.$$
This means that, for every $i=0,1,\ldots,n-1$, either
$$f(\eta_i,\eta_{i+1})=0 \ \ \text{or} \ \ f(\eta_i,-\eta_{i+1})=0, \ \ \text{where} \ \ \eta_0=\eta=\eta_n.$$
Now if $f(\eta_i,\eta_{i+1})=0$ for all $i$, then $\eta$ is a periodic point of $\hat F(x)$. \medskip

Otherwise, there exists an $i$ such that $f(\eta_i,\eta_{i+1}) \neq 0$, but $f(\eta_i,-\eta_{i+1}) = 0$. In this case, if $i<n-1$, replace $\eta_{i+1}$ by $-\eta_{i+1}$ in the next equation of the sequence, yielding $f(-\eta_{i+1},\eta_{i+2}) = 0$.  And if this happens for $i=n-1$, then simply replace $\eta$ by $-\eta$.  This works because $f(-x,y)=f(x,y)$.  In other words, in the chain of equations for $f$, whenever the second argument has a negative sign, choose the next first argument with the same negative sign. And in case the last equation has second argument $\eta$ with a negative sign, then choose the first argument of the first equation as $-\eta$ also.  Hence, there is a chain of equations $f(\eta_i,\eta_{i+1})=0$ beginning and ending with $\pm \eta$. Hence, $\pm \eta$ is a periodic point of $\hat F(x)$ in either case, which implies that $\eta^2$ is the square of a periodic point of $\hat F(x)$.  This completes the proof.
\end{proof}

With this theorem, we have completely proved all the statements in Theorem B of the Introduction.

\noindent Dept. of Mathematical Sciences

\noindent Indiana University -- Purdue University at Indianapolis (IUPUI)

\noindent 402 N. Blackford St., LD 270

\noindent Indianapolis, IN, 46202

\noindent email: pmorton@iupui.edu


\begin{thebibliography}{3}

\bibitem[1]{b1} B. Berndt, {\it Ramanujan's Notebooks Part III}, Springer-Verlag New York Inc. (1991).

\bibitem[2]{b2} B. Berndt, {\it Ramanujan's Notebooks Part V}, Springer-Verlag New York Inc. (1998).

\bibitem[3]{b3} B. Berndt, {\it Number Theory in the Spirit of Ramanujan}, AMS Student Mathematical Library Vol. 34 (2006).

\bibitem[4]{co} D. Cox, {\it Primes of the Form $x^2+ny^2$}, 2nd edition, Wiley, 2013.

\bibitem[5]{d} W. Duke, Continued fractions and modular functions, Bulletin AMS 42 (2) (2005), 137-162.

\bibitem[6]{fk} H. Farkas and I. Kra, {\it Theta constants, Riemann surfaces and the modular group}, American Mathematical Society, Providence, RI, 2001.

\bibitem[7]{fr} W. Franz, Die Teilwerte der Weberschen Tau-Funktion, J. Reine Angew. Math. 173 (1935) 60-64.

\bibitem[8]{g} B. Gordon, Some continued fractions of the Rogers-Ramanujan type, Duke Math. J. 32 (1965), 741-748.

\bibitem[9]{ly} R. Lynch, Arithmetic on Normal Forms of Elliptic Curves, Ph.D. thesis, Indiana University - Purdue University at Indianapolis (IUPUI), granted by Purdue University (2015).

\bibitem[10]{lm} R. Lynch and P. Morton, The quartic Fermat equation in Hilbert class fields of imaginary quadratic fields, Int. J. Number Theory 11 (2015), 1961-2017.

\bibitem[11]{m1} P. Morton, Solutions of diophantine equations as periodic points of $p$-adic algebraic functions, I, New York J. Math. 22 (2016), 715-740.

\bibitem[12]{m} P. Morton, Periodic points of algebraic functions and Deuring's class number formula, Ramanujan J. 50 (2019), 323-354.

\bibitem[13]{m2} P. Morton, Solutions of diophantine equations as periodic points of $p$-adic algebraic functions, II: The Rogers-Ramanujan continued fraction, New York J. Math. 25 (2019), 1178-1213.

\bibitem[14]{m3} P. Morton, Solutions of diophantine equations as periodic points of $p$-adic algebraic functions, III, New York J. Math. 27 (2021), 787-817.

\bibitem[15]{sch} R. Schertz, {\it Complex Multiplication}, New Mathematical Monographs: 15, Cambridge University Press, 2010.

\bibitem[16]{we} H. Weber, {\it Lehrbuch der Algebra, Dritter Band: Elliptische Funktionen und Algebraische Zahlen}, 2nd edition, Friedrich Vieweg und Sohn, Braunshweig, 1908.

\bibitem[17]{yz} N. Yui and D. Zagier, On the singular values of the Weber modular functions, Math. Comp. 66 (1997) 1645-1662.

\end{thebibliography}
\end{document}